\documentclass[10pt, conference]{ieeeconf}      

\IEEEoverridecommandlockouts                              
\usepackage{graphicx} 
\usepackage{epsfig} 
\usepackage{times} 
\usepackage{amsmath} 
\usepackage{amssymb}  
\usepackage{cleveref}
\usepackage{subfigure}
\usepackage{url}
\newtheorem{theorem}{Theorem}

\newtheorem{assumption}[theorem]{Assumption}
\usepackage{color}
\usepackage{mathrsfs}
\usepackage{cite}

\newtheorem{definition}[theorem]{Definition}

\newtheorem{lemma}[theorem]{Lemma}

\newtheorem{proposition}[theorem]{Proposition}

{ 
\newtheorem{remark}[theorem]{Remark}

}
\usepackage{wrapfig}
\usepackage{float}

\newcommand{\bi}{\begin{itemize}}
\newcommand{\ei}{\end{itemize}}
\newcommand{\bd}{\begin{displaymath}}
\newcommand{\ed}{\end{displaymath}}
\newcommand{\be}{\begin{eqnarray*}}
\newcommand{\ee}{\end{eqnarray*}}

\newcommand{\qed}{\nobreak \ifvmode \relax \else
      \ifdim\lastskip<1.5em \hskip-\lastskip
      \hskip1.5em plus0em minus0.5em \fi \nobreak
      \vrule height0.75em width0.5em depth0.25em\fi}

\graphicspath{{Figures/}}

\usepackage{ifthen}
\newboolean{showcomments}
\setboolean{showcomments}{true}

\newcommand{\highlight}[1]{\ifthenelse{\boolean{showcomments}} {\textcolor{red}{#1}}{}}
\newcommand{\edited}[1]{\ifthenelse{\boolean{showcomments}} {\textcolor{blue}{#1}}{}}
\newcommand{\sai}[1]{\ifthenelse{\boolean{showcomments}}
{ \textcolor{red}{(Sai says:  #1)}}{}}
\newcommand{\sinha}[1]{\ifthenelse{\boolean{showcomments}}
{ \textcolor{red}{(Subhrajit says:  #1)}}{}}
\newcommand{\enoch}[1]{\ifthenelse{\boolean{showcomments}}
{ \textcolor{red}{(Enoch says:  #1)}}{}}

\title{Data-Driven Operator Theoretic Methods for Global Phase Space Learning}
\author{Sai Pushpak Nandanoori, Subhrajit Sinha, and Enoch Yeung
\thanks{S. P. Nandanoori and S. Sinha are with Pacific Northwest National Laboratory, Richland, WA 99354 USA and E. Yeung is with University of California Santa Barbara, CA 93106 USA. (emails: saipushpak.n@pnnl.gov, subhrajit.sinha@pnnl.gov,  eyeung@ucsb.edu) This paper is submitted to ACC 2020 and is currently under review. }
}

\begin{document}
\maketitle

\begin{abstract}
In this work, we propose to apply the recently developed Koopman operator techniques to explore the global phase space of a nonlinear system from time-series data. In particular, we address the problem of identifying various invariant subsets of a phase space from the spectral properties of the associated Koopman operator constructed from time-series data. Moreover, in the case when the system evolution is known locally in various invariant subspaces, then a phase space stitching result is proposed that can be applied to identify a global Koopman operator. A biological system, bistable toggle switch and a second-order nonlinear system example are considered to illustrate the proposed results. The construction of this global Koopman operator is very helpful in experimental works as multiple experiments can't be run at the same time starting from several initial conditions. 
\end{abstract}

\section{Introduction}

Dynamics of nonlinear systems, in general, is rich and exhibit various complex behaviours like multiple equilibria, periodic orbits, limit cycles, chaotic attractors etc., and the study of such systems is very challenging. However, the existing methods in nonlinear systems theory require the knowledge of models describing the evolution of nonlinear system and most of the methods require computing an energy function which varies with respect to every nonlinear system. Furthermore, it is very difficult to study higher dimensional systems using model-based approaches. Recent works indicate that these systems can be effectively studied in a higher dimensional linear space using transfer operator, namely Perron-Frobenius and Koopman operators \cite{lasota2013chaos,mezic2005spectral,rowley2009spectral,vaidya_lyapunov_measure,mezic2013analysis, yeung2017learning}. These operators are adjoint of one another and thus, in principle, a dynamical system can be studied equivalently using either of the operators. However, as far as applicability to real life scenarios is concerned, each of the operators has their own pros and cons. For example, in data-driven analysis, Koopman operator is more suited compared to the Perron-Frobenius operator. Koopman operator was originally introduced by Bernard Koopman in \cite{koopman1931hamiltonian} and this seminal work became a popular tool with \cite{mezic2005spectral} and several other works followed where Koopman operator is used in system identification \cite{rowley2009spectral,schmid2010dynamic}, control design \cite{proctor2016dynamic,brunton2016koopman,huang2018feedback} and finding observability gramians/observers \cite{vaidya2007observability,surana2016linear}.

Koopman operator is an infinite dimensional operator that can represent the evolution of a dynamical system. However, finding an infinite dimensional (linear) operator is computationally intractable and numerous methods have been developed to best represent the (nonlinear) system dynamics with a finite dimensional Koopman operator \cite{rowley2009spectral,tu2013dynamic,williams2014kernel,williams2015data,sinha_robust_acc,sinha_robust_arxiv,sinha_sparse_acc}. Most popular methods include, dynamic mode decomposition (DMD) \cite{rowley2009spectral,tu2013dynamic}; extended dynamic mode decomposition (E-DMD) \cite{williams2015data,li2017extended}; kernel dynamic mode decomposition (K-DMD) \cite{williams2014kernel}, naturally structured dynamic mode decomposition (NS-DMD) \cite{huang2017data}, deep dynamic mode decomposition (deep-DMD)\cite{yeung2017learning,lusch2018deep}. All these methods are based on data and accuracy of such approximations are discussed in \cite{zhang2017evaluating}. In \cite{johnson2017class}, the authors study the space of Koopman observables which also includes the states of the underlying system. The authors of \cite{johnson2017class} also show the fidelity of this class of observable functions and discuss ways to tune performance.  


\subsection{Problem Statement}
The main focus of this work is to look at an alternative approach to explore the phase space of nonlinear systems from time-series data. In particular, given the time-series data for a nonlinear system, it is of interest to know how many invariant manifolds the system has or if the system trajectories start from an invariant subspace, can the other invariant subspaces be identified? Further, if the behavior of a system in subspaces is known, then can these invariant subspaces be stitched to identify the evolution of system on global phase space? This work attempts to address these questions by leveraging the tools from Koopman operator theory. This work is motivated from the phase space analysis of the nonlinear system, bistable toggle switch which has two stable attractors and the system exhibits bistability. 

\subsection{Summary of Contributions}

A nonlinear system may have multiple invariant sets. In this paper, we use the Koopman operator techniques to identify the different invariant subspaces from a global Koopman operator. This can be thought of as a top-down approach, where we use the global Koopman operator to zoom in on the state space and explore local dynamics. On the other hand, we also propose a bottom-up construction technique, where we use local Koopman operators to obtain a global Koopman operator for the system on the entire state space. This phase space stitching algorithm is of practical importance in the sense that often in experiments, one obtains time-series data from the different invariant subspaces and thus having multiple Koopman operators, each corresponding to a different invariant subspace. So if we need to predict the future trajectory from a new hitherto unseen initial condition, we may not know which of the different Koopman operators to use, for propagating this initial condition. This is because we may not know which invariant set the given initial condition belongs to. Our phase space stitching algorithm provides a technique to \emph{stitch}  multiple \emph{local} Koopman operators, to obtain a single global Koopman operator so that one can propagate any initial condition, without prior knowledge of the invariant set where the initial condition belongs to.
The phase space learning problem from this work serves as a basis for designing and planning future experiments. 
 

The rest of the paper is organized as follows. Section \ref{sec:math_preliminaries} describes the preliminaries for operator theoretic methods. A brief overview on the computation of finite-dimensional Koopman operator is given in Section \ref{sec:DMD_EDMD}. Section \ref{sec:identify_inv_spaces} discusses discovering new invariant subspaces from the spectral properties of the Koopman operator. Section \ref{sec:state_space_partition} describes, using ergodic properties of the dynamical system, identification of the invariant subspaces. A phase space stitching algorithm is presented in Section \ref{sec:phase_space_stitching} to compute the Koopman operator corresponding to the complete phase space. Example studies corresponding to two nonlinear systems are presented in Section \ref{sec:simulation} and the paper concludes with final remarks in Section \ref{sec:conclusion}. 



\section{Mathematical Preliminaries and Problem Statement}
\label{sec:math_preliminaries}
We discuss only the necessary mathematical background in this section. More technical details can be found in \cite{lasota2013chaos,mezic2005spectral,mezic2013analysis,budivsic2012applied,williams2015data} and the references within. Consider the following discrete-time nonlinear system
\begin{align}
    z_{t+1} = T(z_t)
    \label{eq:DT_NL_sys}
\end{align}
where $z \in {\cal M}$ and $T:{\cal M} \to {\cal M}$. The phase space ${\cal M}$ is assumed to be a compact Riemannian manifold with Borel $\sigma$ algebra $\mathcal{B({\cal M})}$ and measure $\mu$ \cite{mezic2005spectral} and the map $T$ is assumed to be a diffeomorphism. 
Without loss of generality, $T$ can be assumed to be ergodic, since if $T$ is not ergodic, the phase space can be partitioned into ergodic partitions, where in each of the partitions $T$ is ergodic.


Define a function ${\psi}:\mathcal{M} \rightarrow \mathbb{C}$ to be an observable, where ${\psi}$ is an element of the space of functions ${\cal G}$ acting on elements of $\cal M$. In general, for any given dynamical system, the evolution of state dynamics is monitored by the output function and this output function may be considered as an observable function. However, any \emph{nicely behaved} (to be made precise) function of the states can be an observable function. Normally, the space of square integrable functions on $\cal M$ is taken to be the space of observables. Hence the dimension of the function space ${\cal G}$ is infinite, with ${\cal G} = L_2({\cal M}, {\cal B}, \mu)$. With this, the Koopman operator is defined as follows: 
\begin{definition}[Koopman operator] For a dynamical system $x\mapsto T(x)$ and for ${\psi} \in \mathcal{G}$, the Koopman operator $({\bf U} : {\cal G} \to {\cal G})$ associated with the dynamical system is defined as
\[ [\mbox{\textbf U} {\psi}] (x) = \psi(T(x)). \]
\end{definition}
Notice that the function space, ${\cal G}$ is invariant under the action of Koopman operator which is a linear operator and it describes the evolution of the observables. Hence the Koopman operator defines a linear system on the space of functions. In other words, instead of studying the time evolution of system \eqref{eq:DT_NL_sys} over the state space ${\cal M}$ which may be a nonlinear evolution, we look at the linear evolution of \textit{observables} in a higher dimensional space.


Apart from being linear, the Koopman operator satisfies the following properties:
\begin{itemize}
\item Koopman operator is unitary: given ${\cal G} = L_2({\cal M}, {\cal B}, \mu)$, and  the Koopman operator $\mbox{\textbf{U}}: L_2({\cal M}, {\cal B}, \mu) \to L_2({\cal M}, {\cal B}, \mu)$ is unitary, that is, 
\begin{align*}
    \parallel \mbox{\textbf{U}} h \parallel^2= & \int_{\cal M} |h(T(z))|^2d \mu(z)\\
    =&\int_{\cal M} | h(z)|^2 d\mu(z)=\parallel h\parallel^2
\end{align*}
and hence all the eigenvalues of $\mbox{\textbf{U}}$ lie on an unit circle and its eigenfunctions are orthogonal.  
\item Koopman operator is a positive operator: for any $\psi \geq 0$, $[\mbox{\textbf{U}} \psi] (x) \geq 0$.  
\end{itemize}

Similar properties can be arrived for the Koopman operator corresponding to continuous-time dynamical systems. In the next section, we briefly describe the popular DMD, EDMD algorithms that are used to compute a finite dimensional approximation of the Koopman operator from time-series data. 

\section{Dynamic Mode Decomposition (DMD) and Extended Dynamic Mode Decomposition (EDMD)}
\label{sec:DMD_EDMD}
Schmid and others introduced DMD in \cite{schmid2010dynamic} to approximate the Koopman operator that describes the coherent features of fluid flow and extract dynamic information from flow fields through the spectrum of the approximate Koopman operator. The authors in \cite{williams2015data} generalized the idea of DMD and extended it to what is now known in literature as EDMD. The approximate Koopman operator computed based on EDMD evolves in the space of finite set of observable functions. In this section, we briefly recall the EDMD algorithm and discuss the computation of approximate Koopman operator from time-series data. 

Consider the time-series data $[x_0,x_1,\cdots , x_k]$ from an experiment or simulation of a dynamical system and stack them as shown below. 
\begin{align*}
X_p = [x_0, x_1, \dots, x_{k-1}], \qquad X_f = [x_1,x_2, \dots,x_k],
\end{align*}
where for every $i$, $x_i \in {\cal M}$. Denote the set of dictionary functions or observables by ${\cal D} = \{\psi_1,\dots,\psi_N\}$ where $\psi_i \in L_2({\cal M},{\cal B},\mu)$ and $\psi_i:{\cal M} \to \mathbb{C}$. Let the span of these $N$ observables be denoted by ${\cal G}_{\cal D}$ such that ${\cal G}_{\cal D} \subset {\cal G}$ where ${\cal G} = L_2({\cal M}, {\cal B}, \mu)$. 


Define a vector valued observable function $\pmb{\Psi}:{\cal M} \to \mathbb{C}^N$ as follows. 
\begin{align*}
\pmb{\Psi}(x):=\begin{bmatrix}\psi_1(x) & \psi_2(x) & \cdots & \psi_N(x)\end{bmatrix}^\top.
\end{align*}
%
Let $\phi, \hat{\phi} \in {\cal G}_{\cal D}$ be any two functions in the space of observables. Then they can be expressed as 
\begin{align*}
\phi = \sum_{k=1}^N a_k\psi_k=\pmb{\Psi}^T \pmb{a},\qquad \hat{\phi} = \sum_{k=1}^N \hat{a}_k\psi_k=\pmb{\Psi}^T {\hat{\pmb{a}}}
\end{align*}
for some set of coefficients $\pmb{a},{\hat{\pmb{a}}}\in \mathbb{C}^N$. Then the functions, $\phi$ and $\hat{\phi}$ can be related through the Koopman operator as shown below. 
\[ \hat{\phi}(x)=[\mbox{\textbf{U}}\phi](x)+r,\]
where $r$ is a residual function that appears because $\mathcal{G}_{\cal D}$ is not necessarily invariant to the action of the Koopman operator. Minimizing the residue function $r$ is posed as a least square problem as shown below. 
%
%
\begin{equation}
\min_{\cal K}\parallel{\cal K} {Y_p}-{Y_f}\parallel_F,
\label{edmd_op}
\end{equation}
where ${\cal K}$ is the finite dimensional approximation of the Koopman operator $\mbox{\textbf{U}}$ and 
%
\begin{align*}
& {Y_p}=\pmb{\Psi}(X_p) = [\pmb{\Psi}(x_0), \pmb{\Psi}(x_1), \cdots , \pmb{\Psi}(x_{k-1})]\\
& {Y_f}=\pmb{\Psi}(X_f) = [\pmb{\Psi}(x_1), \pmb{\Psi}(x_2), \cdots , \pmb{\Psi}(x_k)],
\end{align*}
%
with ${\cal K}\in\mathbb{C}^{N\times N}$. The solution to the optimization problem \eqref{edmd_op} can be obtained explicitly and the approximate Koopman operator is given by 
\begin{align*}
{\cal K}={Y_f}{Y_p}^\dagger 
\end{align*}
where ${Y_p}^{\dagger}$ is the pseudo-inverse of matrix $Y_p$. Note that, DMD is a special case of EDMD algorithm with $\pmb{\Psi}(x) = x$. Once the approximate Koopman operator is obtained, the spectral properties of Koopman operator and in particular, the dominant eigenvalues and their corresponding eigenfunctions are used for analyzing the evolution of nonlinear systems. However, one constraint is that the dictionary functions should be rich enough to approximate the Koopman spectrum to desired degree of accuracy.

\section{Global Phase Space Exploration Using Approximate Koopman Operator}
This section describes how from the spectral properties of the Koopman operator, the global phase space of any dynamical system can be explored. In particular, the results in this section can be effortlessly applied to identify multiple invariant subspaces in a phase space and to describe the global behavior of the phase space assuming the local behaviors of the system are known. 

\subsection{Discovering Multiple Invariant Subspaces Using Koopman Spectra}
\label{sec:identify_inv_spaces}
A Koopman operator is an exact representation of a dynamical system on an infinite dimensional Hilbert space and the eigen spectrum and eigenfunctions of a Koopman operator contains all the information about the underlying dynamical system. More specifically, we start in an invariant subspace by computing a local Koopman operator and with addition of new data points, we update the Koopman operator and eventually compute the Koopman operator for the entire state space. An inherent assumption being made here is that, eventually we get the time-series data from all the invariant subspaces in the system. Otherwise, it is intractable to construct the global Koopman operator. Recall that the Koopman operator is unitary and it has eigenvalues on the unit circle.
%
%
%
%
\begin{lemma}
Let  ${\cal M}_p\subset {\cal M}$, for $p=1,\cdots ,v$ be the $v$ invariant subsets of a dynamical system $x\mapsto T(x)$, where ${\cal M}_p\cap {\cal M}_q=\phi$ for $p\neq q$. Let ${\cal M}_p \subset {\cal M}$  be a local invariant subspace and assume $^pX$ be the time-series data from this subspace. Let ${\cal K}_{p}$ be the local Koopman operator trained on $^pX$ such that
%
\begin{align}
    \max_{x_{0} \in {\cal M}_p} || \mathcal{K}_{p}^n\pmb{\Psi}(x_0) - \pmb{\Psi}(T^n(x_0))|| \leq \varepsilon_{p} 
    \label{eq_Koopman_learning_error}
\end{align}
where $\varepsilon_{p}>0$ is a small positive number and $n$ is the number of time steps for which the initial condition $x_0$ is evolved. The  notation, ${\cal K}_p^n$ indicates that Koopman operator ${\cal K}_p$ is raised to the power $n$. 
%
Then, given a subspace ${\cal M}_q$ such that ${\cal M}_p \subset {\cal M}_q \subseteq {\cal M}$ 
the Koopman learning error on the new data-set $X_q$ satisfies:
\begin{equation}
    \begin{aligned}
    & \min_{x_{0} \in {\cal M}_q} ||\mathcal{K}_{p}^n\pmb{\Psi}(x_{0})  - \pmb{\Psi}(T^n(x_0)) ||  \\
    & \geq \max_{x_{0} \in {\cal M}_p} || \mathcal{K}_{p}^n \pmb{\Psi}(x_{0}) - \pmb{\Psi}(T^n(x_0))||. 
    \end{aligned}
    \label{eq:K_error_ineq}
\end{equation}
\label{lemma_K_learning_error}
\end{lemma}
\begin{proof}
We argue that the proof follows by contradiction. Suppose, Eq. \eqref{eq:K_error_ineq} does not hold. Then from Eq. \eqref{eq_Koopman_learning_error}, we have:
\begin{equation}
    \begin{aligned}
    & \min_{x_{0} \in {\cal M}_q} ||\mathcal{K}_{p}^n \pmb{\Psi}(x_{0}) - \pmb{\Psi}(T^n(x_0)) || \\
    & <  \max_{x_{0} \in {\cal M}_p}|| \mathcal{K}_{p}^t \pmb{\Psi}(x_{0}) - \pmb{\Psi}(T^n(x_0))|| \leq \varepsilon_{p}
    \end{aligned}
\end{equation}
This essentially means that, the Koopman operator, ${\cal K}_{p}$ trained on the time-series dataset from ${\cal M}_p$ satisfies the time-series data from ${\cal M}_q$ as well. This implies that, ${\cal M}_q \subseteq {\cal M}_p$ which leads to a contradiction. Therefore, Eq. \eqref{eq:K_error_ineq} holds and hence the proof. 
\end{proof}

From Lemma \ref{lemma_K_learning_error}, it can be said that for any time-series dataset from a subspace which is a superset of ${\cal M}_p$, if the Eq. \eqref{eq:K_error_ineq} holds, then the Koopman operator has to be recomputed corresponding to the new time-series dataset. The following result formally establishes the discovery of an invariant subspace based on the spectral properties of the approximate local Koopman operators.  

\begin{proposition}
Let the subspace ${\cal M}_p \subset {\cal M}$ be the smallest invariant subspace of ${\cal M}$ and the corresponding Koopman operator is given by ${\cal K}_{p}$. Choose ${\cal M}_q$ such that ${\cal M}_p \subset {\cal M}_q \subseteq {\cal M}$ and under the scenario where \eqref{eq:K_error_ineq} holds, the Koopman operator corresponding to ${\cal M}_q$ is given by ${\cal K}_q$. 
Then the discovery of new invariant subspaces in the phase space of the original dynamical system is such that the difference $||{\cal K}_{q}||_{\rho} - ||{\cal K}_{p}||_{\rho}$ increases monotonically, where $\parallel {\cal K} \parallel_{\rho}$ denote the geometric multiplicity of eigenvalue, $\lambda = 1$, that is, number of distinct eigenvectors associated with the stable eigenvalue $\lambda = 1$. 
\end{proposition}
\begin{proof}
The proof follows by noting that for a subspace ${\cal M}_q \supset {\cal M}_p$, the inequality $\parallel {\cal K}_q \parallel_{\rho} \geq \parallel {\cal K}_p \parallel_{\rho}$ is always satisfied.
\end{proof}
%
%

In the next section, we show how the complete phase space can be separated into ergodic partitions (invariant subspaces) from the spectral properties of the Koopman operator. 
%
%
\subsection{Invariant Phase Space Decomposition Using Global Koopman Operator}
\label{sec:state_space_partition}
In this section we discuss how can the global Koopman operator be used to explore a dynamical system \emph{locally}. We recall some relevant concepts from \cite{petersen1989ergodic,budivsic2012applied} for the self-containment of this paper. Although we present the results below in terms of a finite  dimensional (approximate) Koopman operator, they are indeed general and are applicable to infinite dimensional Koopman operators. Interested readers can refer \cite{petersen1989ergodic,budivsic2012applied} and the references within for detailed explanation.

%


Let $T$ be a measure preserving transformation and ${\cal G} = L_2({\cal M}, {\cal B}, \mu)$, then all the eigenvalues of the associated Koopman operator ${\cal K}$ lie on the unit circle \cite{budivsic2012applied}. 
Moreover, when $T$ is an ergodic transformation, that is, for any set ${\cal S} \subset {\cal M}$, such that $T^{-1}({\cal S}) = {\cal S}$, either $\mu({\cal S}) = 0$ or $\mu({\cal S}) = 1$, then all eigenvalues of ${\cal K}$ are simple \cite{petersen1989ergodic,budivsic2012applied}. However, if $T$ is not ergodic, then the state space can be partitioned into subsets ${\cal S}_i$ (minimal invariant subspaces) such that the restriction $T|_{{\cal S}_i}: {\cal S}_i \to {\cal S}_i$ is ergodic. A partition of the state space into invariant sets is called an ergodic partition or stationary partition.

Furthermore, all ergodic partitions are disjoint and they support mutually singular functions from ${\cal G}$ \cite{budivsic2012applied}. Therefore, the number of linearly independent eigenfunctions of ${\cal K}$ corresponding to an eigenvalue $\lambda$ is bounded above by the number of ergodic sets in the state space \cite{budivsic2012applied}. The dynamics of the system dictates number of ergodic partitions (invariant sets) in the state space. The following results summarize the above discussion. 
\begin{lemma}
Let $\lambda$ be an eigenvalue of the Koopman operator ${\cal K}$.  Suppose if the algebraic multiplicity of the eigenvalue $\lambda$, is equal to the geometric multiplicity, then the corresponding eigenfunctions are linearly independent. 
\label{lemma:am_gm} 
\end{lemma}
\begin{proof}
The proof follows from standard results on matrices \cite{horn2012matrix}. 
\end{proof}
\begin{lemma}
Let ${\cal K}$ be the Koopman operator and its corresponding eigenvalue $\lambda$ has an algebraic multiplicity of $p$. Then $p$ linearly independent functions map to at most $p$ invariant subspaces in the state space.  
\label{lemma:finding_invariant_subspaces}
\end{lemma}
\begin{proof}
The result follows from \cite{budivsic2012applied}.
\end{proof}
Suppose ${\cal M}_p$ is an invariant subspace of ${\cal M}$. Let ${\cal K}_p$ be the corresponding Koopman operator acting on the vector valued observable denoted by $\pmb{\Psi}_p$ which is obtained by stacking all the observable functions in ${\cal M}_p$. We claim that any point $x$ outside of ${\cal M}_p$ will not satisfy the local Koopman invariant subspace equation with minimal error. In particular, Eq. \eqref{eq_Koopman_learning_error} does not hold and we obtain:
\begin{align*}
    \parallel {\cal K}_{p}^n \pmb{{\Psi}}_{p}(x) - \pmb{\Psi}_{p}(T^n(x)) \parallel > \varepsilon_{p} \; \mbox{for} \; x \in {\cal M}\setminus{\cal M}_p.     
\end{align*}

The following result describes how the Koopman operator corresponding to each of the local invariant subspace can be identified from the global Koopman operator. 
\begin{proposition}
Let ${\cal K}$ denote the global Koopman operator for a dynamical system \eqref{eq:DT_NL_sys}. Suppose every eigenvalue of ${\cal K}$ satisfies that its algebraic multiplicity is equal to its geometric multiplicity, ${\cal K}$ can be block diagonalized, where each block corresponds to a local invariant Koopman operator. 
\label{prop:blk_diagoanl_K}
\end{proposition}
\begin{proof}
The result follows from Jordan block decomposition of a
matrix.
\end{proof}

Note that each local Koopman operator corresponds to an invariant subspace, which is an element of the ergodic partition of the state space. 
The next section describes how the evolution of system on the complete phase space (ie., a global Koopman operator) can be identified if the system evolution on invariant subspaces is known (ie., local Koopman operators are known). 
\subsection{Coupling Distinct Phase Representations Assuming Approximate Koopman Operators}
\label{sec:phase_space_stitching}
Let $x\mapsto T(x)$ with $x\in {\cal M}$ and let ${\cal M}_p$, $p=1,\cdots , v$ be the invariant sets of the dynamical system. For any ${\cal M}_p$, let 
\[{\bf \Psi}_{p} = \{\psi_1^p, \cdots , \psi_{n_p}^p\}\]
be the set of dictionary functions used for computation of the local Koopman operator ${\cal K}_p$. Then these local Koopman operators can be combined to form a single Koopman operator which we refer to as the \textit{stitched Koopman operator} and it is given by 
\begin{align*}
     {\cal K}_{\cal S} = \mbox{diag}({\cal K}_{1},{\cal K}_2,\cdots,{\cal K}_{v})
    \end{align*}
 with 
 \begin{align*}
    \pmb{\Psi}(x) = \begin{bmatrix} \chi_1(x) {\bf \Psi}_{1}(x) \\ \chi_2(x) {\bf \Psi}_{2}(x) \\ \vdots \\ \chi_v(x) {\bf \Psi}_{v}(x) \end{bmatrix}
\end{align*}
where $\chi_p(x)$ is the characteristic function corresponding to the (invariant) subspace ${\cal M}_p$ and it is defined as follows for $p=1,\cdots,v$. 
\begin{align*}
    \chi_p(x) = \begin{cases} 1, & \mbox{if}\;\ x \in {\cal M}_p \\ 
    0, & \mbox{otherwise}. \end{cases}
\end{align*}
%
The order of these local Koopman operators while forming the stitched Koopman operator doesn't matter as long as the local Koopman operators are stacked in accordance with their corresponding observable functions. Note that the intersection of any two invariant sets is a set of measure zero and such sets have not been considered in the above result. The stitched Koopman operator is a block diagonal matrix with ${\cal K}_{\cal S}\in \mathbb{R}^{L\times L}$ where $L = \sum_{p=1}^vn_p$.
One of the key differences between identifying the global Koopman operator as described in Section \ref{sec:identify_inv_spaces} and \ref{sec:phase_space_stitching} is summarized in the remark given below. 
\begin{remark}
In Section \ref{sec:identify_inv_spaces}, the global Koopman operator is identified by starting locally in an invariant subspace and with new time series data whenever \eqref{eq:K_error_ineq} is satisfied, the Koopman operator is updated. This process is repeated until the global Koopman operator is identified. This is however different from the (global) stitched Koopman operator which assumes the knowledge of Koopman operators and dictionary functions in each of the invariant subspaces. Section \ref{sec:phase_space_stitching} results can be readily applied for time-series data generated by multiple experiments. 
\end{remark}

\section{Simulation Results}
\label{sec:simulation}
The proposed phase space stitching results are illustrated on two different systems. 
We first compute the global Koopman operator ${\cal K}$ assuming the knowledge of time-series data from the entire state space. Next, we assume the knowledge of time-series data only in invariant subspace and compute local Koopman operator respective to each invariant subspace. These local Koopman operators are then used to compute the stitched Koopman operator ${\cal K}_{\cal S}$ applying the results from Section \ref{sec:phase_space_stitching}. We finally establish that the attractor sets are well captured by ${\cal K}_{\cal S}$. We first begin with the bistable toggle switch system. 
\subsection{Bistable Toggle Switch}
Bistable toggle switch is governed by the following equations: 
\begin{equation}
\begin{aligned}
\dot{x}_1 = & \frac{\alpha_1}{1+x_2^{\beta}} - \kappa_1 x_1 \\
\dot{x}_2 = & \frac{\alpha_2}{1+x_1^{\gamma}} - \kappa_2 x_2
\end{aligned}
\label{eq:bistable_toggle_switch}
\end{equation}
where $x_1 \in \mathbb{R}$ and $x_2 \in \mathbb{R}$ represent the concentration of the repressor $1$ and $2$; $\alpha_1$ and $\alpha_2$ denote the effective rate of synthesis of repressor $1$ and $2$; $\kappa_1>0$ and $\kappa_2>0$ are the self decay rates of concentration of repressor $1$ and $2$; $\beta$ and $\gamma$ denote the cooperativity of repression of promoter $2$ and $1$. This bistable toggle switch is a biological system and its mathematical description was first introduced in \cite{gardner2000construction}. This system exhibits bistability, that is, this system has two equilibrium points that are stable and the domain of attraction corresponding to these stable equilibrium points are separated by a separatrix. The phase portrait of this system with two stable equilibrium points (i.e., two stable attractors) can be seen in Fig. \ref{fig:bistable_phase_portrait}. It is important to recall that if the parameters (such as $\alpha_1, \alpha_2$) are different for gene concentrations, $x_1, x_2$, a monostable equilibrium is seen \cite{gyorgy2016quantifying}.  
\begin{figure}
\begin{center}
\includegraphics[width = 0.95 \linewidth]{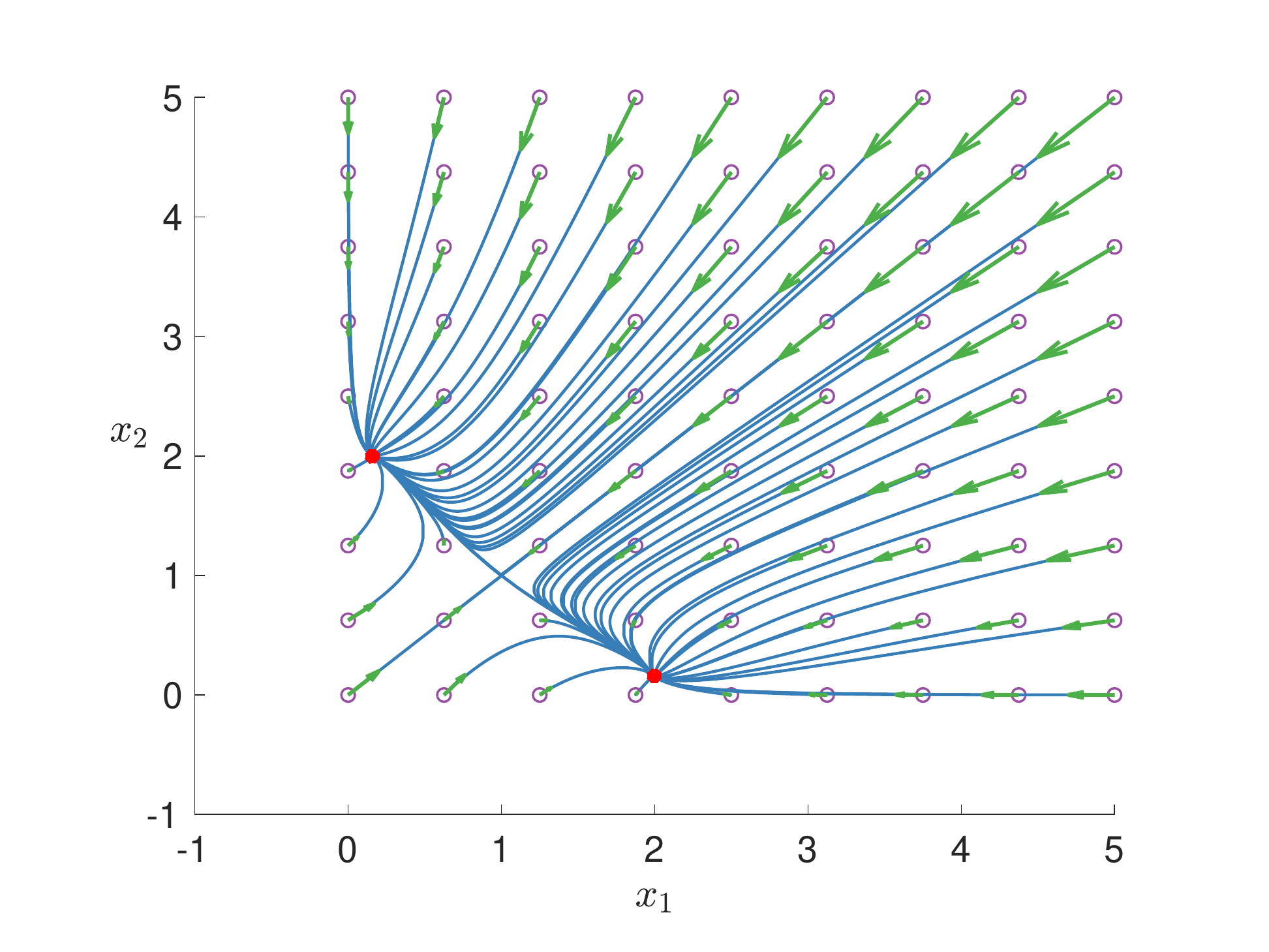}
\caption{Phase portrait of bistable toggle switch. This system has two stable attractors for device parameters: $\beta = 3.55$, $\gamma = 3.53$, $\alpha_1 = \alpha_2 = 1$ and $\kappa_1 = \kappa_2 = 0.5$. Circles indicate the initial conditions, green arrow indicate the direction of the vector field at the initial condition and the red dots are the equilibrium points.}
\label{fig:bistable_phase_portrait}
\end{center}
\end{figure}
%
%
We considered $1000$ time points for $81$ initial conditions in the entire state space (from both the invariant sets) and these data points are lifted to a higher dimensional space using the Gaussian radial basis functions. The associated Koopman operator is computed using the EDMD algorithm by minimizing the residue function $r$ as shown in Eq. \eqref{edmd_op}. The size of the dictionary is chosen to be $30$ and each dictionary function is of the form, $\psi(x) = \exp(-{\parallel x \parallel^2/\sigma^2})$ where $\sigma = 0.4$. The corresponding eigenvalues of the global Koopman operator ${\cal K}$ are shown in Fig. \ref{fig:eigenvalues_bistable}. 
%
%
\begin{figure}[h!]
\begin{center}
\includegraphics[width = 0.8 \linewidth]{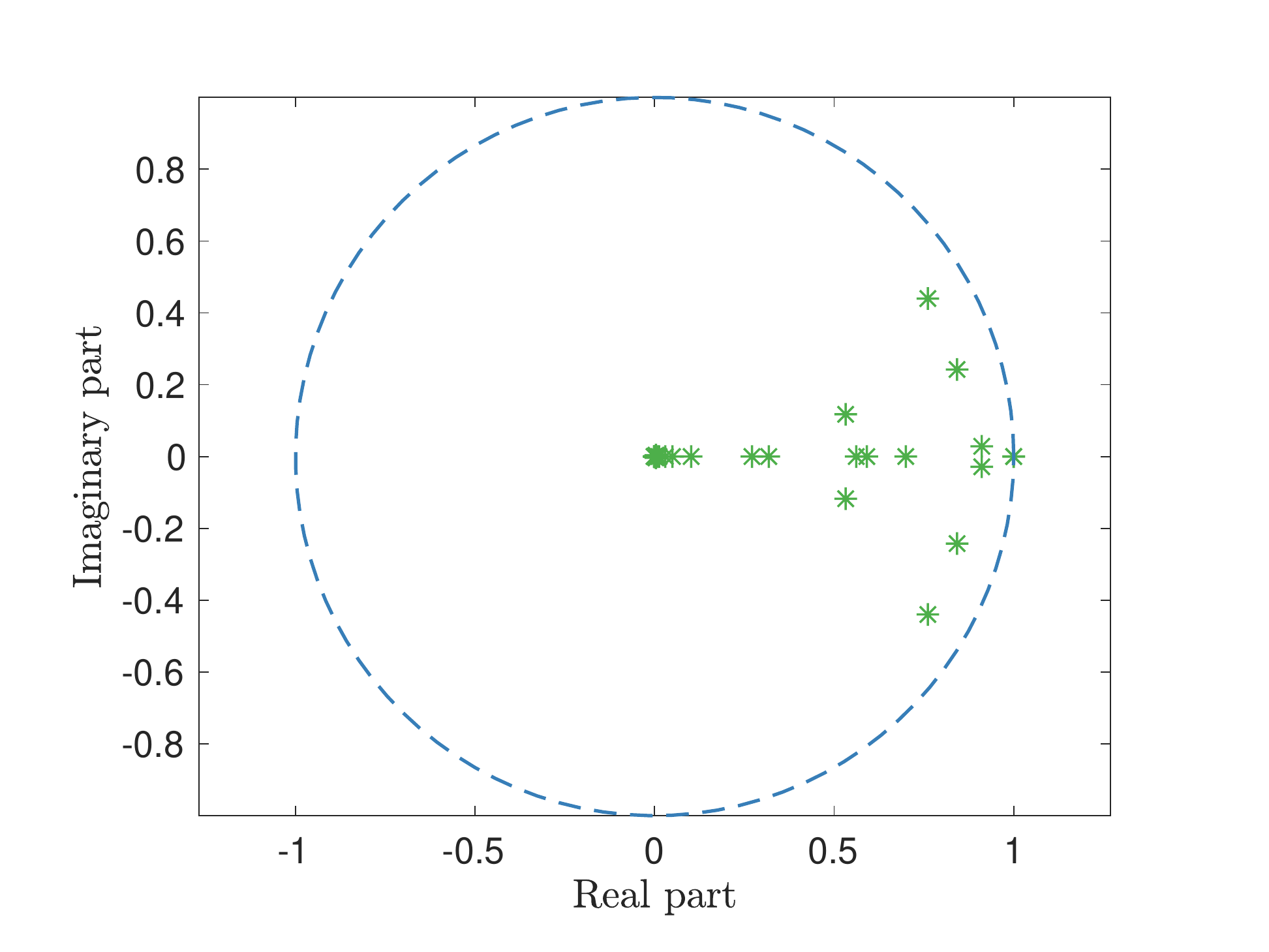}
\caption{Eigenvalues of the Koopman operator ${\cal K}$ corresponding to the bistable toggle switch system given in Eq. \eqref{eq:bistable_toggle_switch}.}
\label{fig:eigenvalues_bistable}
\end{center}
\end{figure}
%
%

The equilibrium points for the bistable toggle switch system are given by $(2,0.16)$ and $(0.16, 2)$. The Koopman operator ${\cal K}$ has two (dominant) eigenvalues, $\lambda = 1$ and these two eigenvalues correspond to the two equilibrium points and it can be seen by plotting the eigenvectors associated with the dominant eigenvalues on the state space as shown in Fig. \ref{fig:bistable_inv_spaces}. It can be seen that, clearly these two eigevectors capture the attractor sets (equilibrium points) in the state space. 
\begin{figure}[h!]
\begin{center}
\subfigure[]{\includegraphics[width = 0.46 \linewidth]{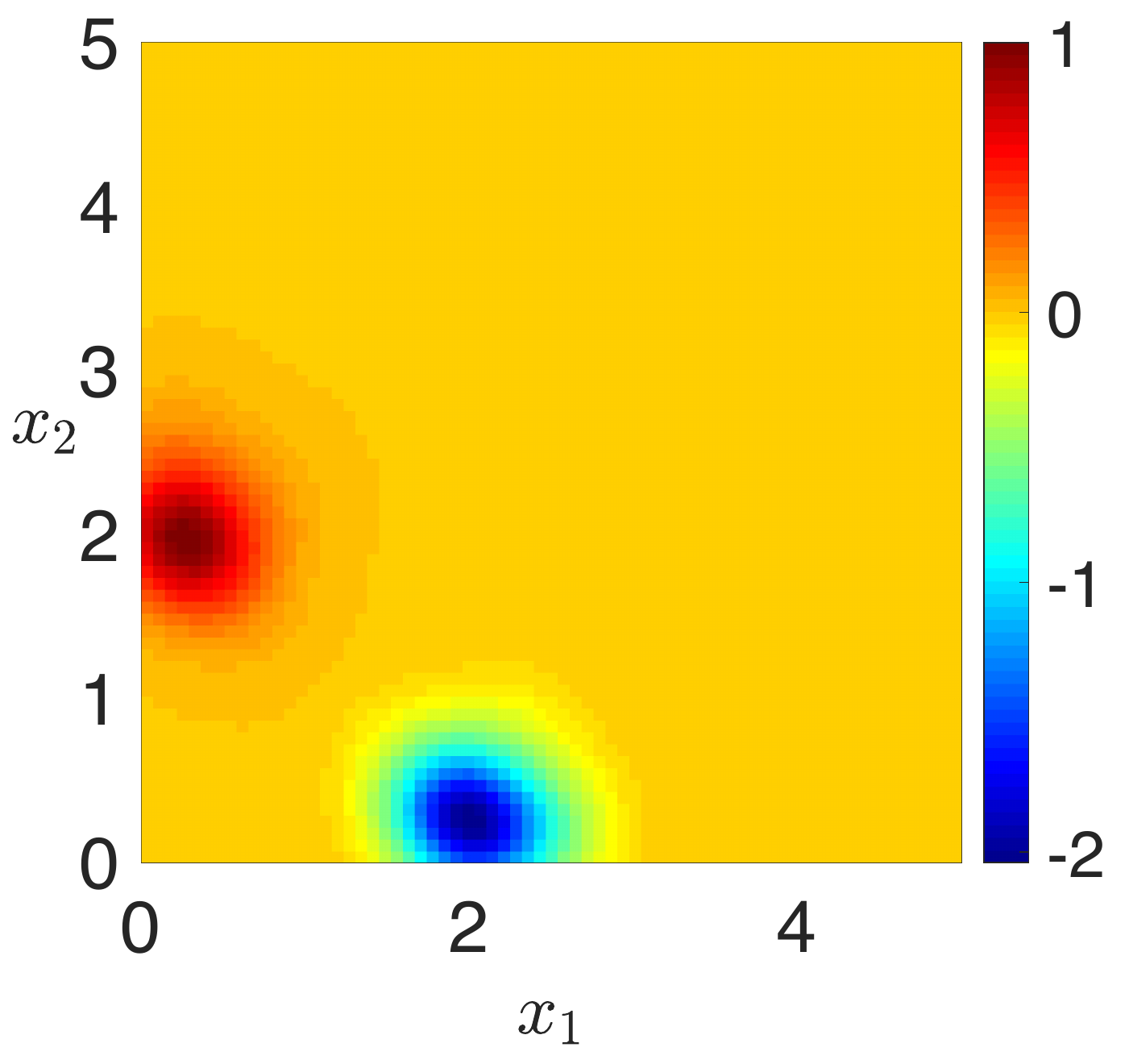}}
\subfigure[]{\includegraphics[width = 0.46 \linewidth]{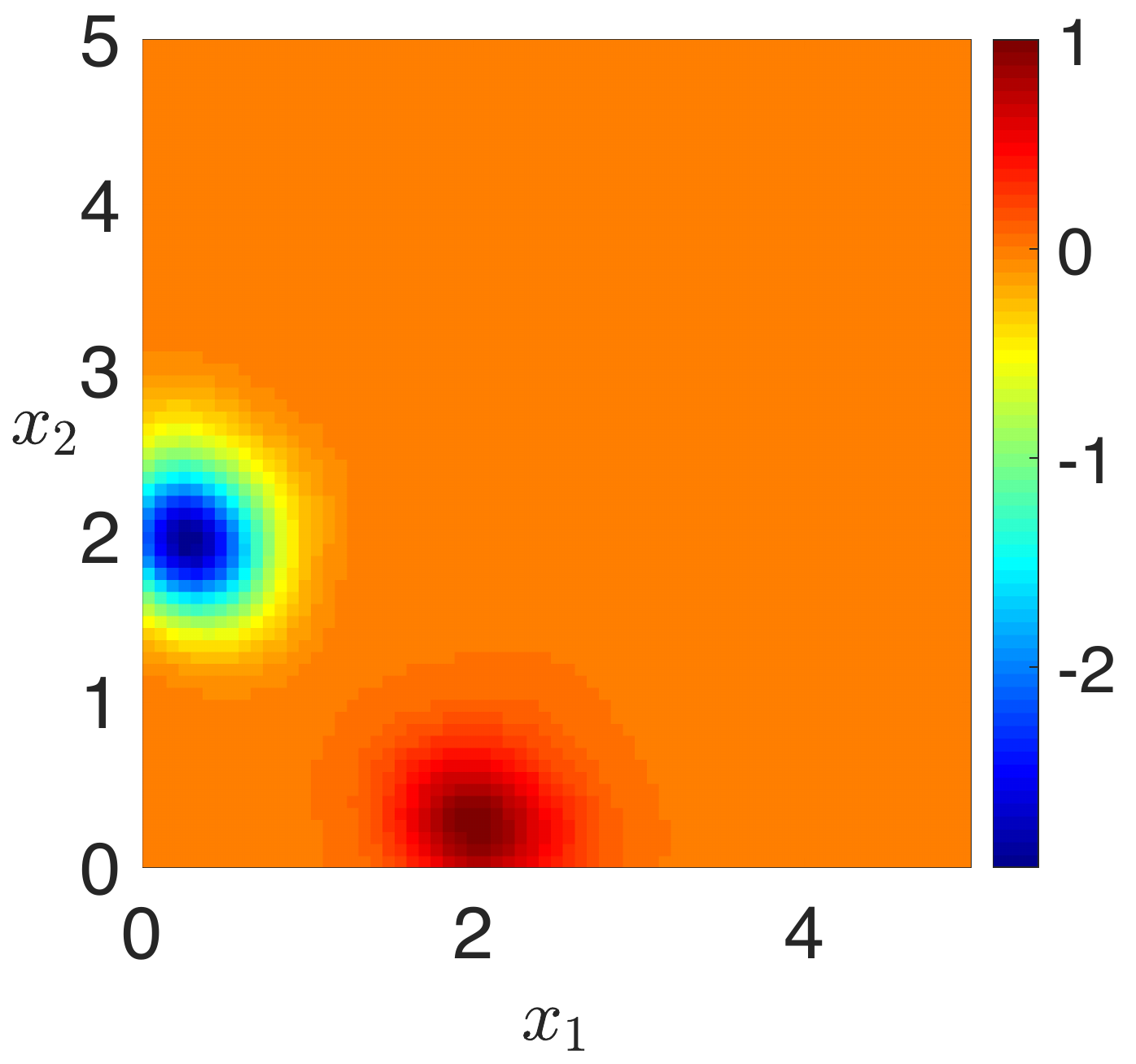}}
\caption{(a) and (b) Eigenvectors of the Koopman operator corresponding to the eigenvalue $\lambda = 1$ on the state space. The region around the equilibrium point can be seen inside the (blue) colored ellipses.}
\label{fig:bistable_inv_spaces}
\end{center}
\end{figure}

Next, assuming the knowledge of the entire phase space is not known, we consider the data on each side of the separatrix at a time and compute the corresponding local Koopman operators. Using these local Koopman operators, we demonstrate the phase space stitching result by computing their equivalent (stitched) global Koopman operator. For the computation of Koopman operator on each attractor, again $30$ Gaussian radial basis functions are used. Let's denote the respective local Koopman operators by ${\cal K}_{l}$ and ${\cal K}_r$. It can be seen that each of these local Koopman operators has one eigenvalue $\lambda = 1$ and their corresponding eigenvectors correspond to the invariant regions in the state space. 
%

As we now have the knowledge of the local invariant subspaces, these spaces can be stitched and the evolution of the system on the complete phase space can be given by a single Koopman operator (as described in Section \ref{sec:phase_space_stitching}), which we are referring to as the stitched Koopman operator, ${\cal K}_{\cal S}:=$diag$({\cal K}_l, {\cal K}_r)$. The eigenvalues of the stitched Koopman operator are shown in Fig. \ref{fig:bistable_eigenvalues_stitched}. It is important to note that the size of ${\cal K}_{\cal S}$ is $60\times 60$ whereas the size of ${\cal K}$ is $30\times 30$. Moreover the sparse structure of both of these Koopman operators can be seen in Fig. \ref{fig:bistable_Koopman_structures}. 
\begin{figure}[h!]
\begin{center}
\includegraphics[width = 0.8 \linewidth]{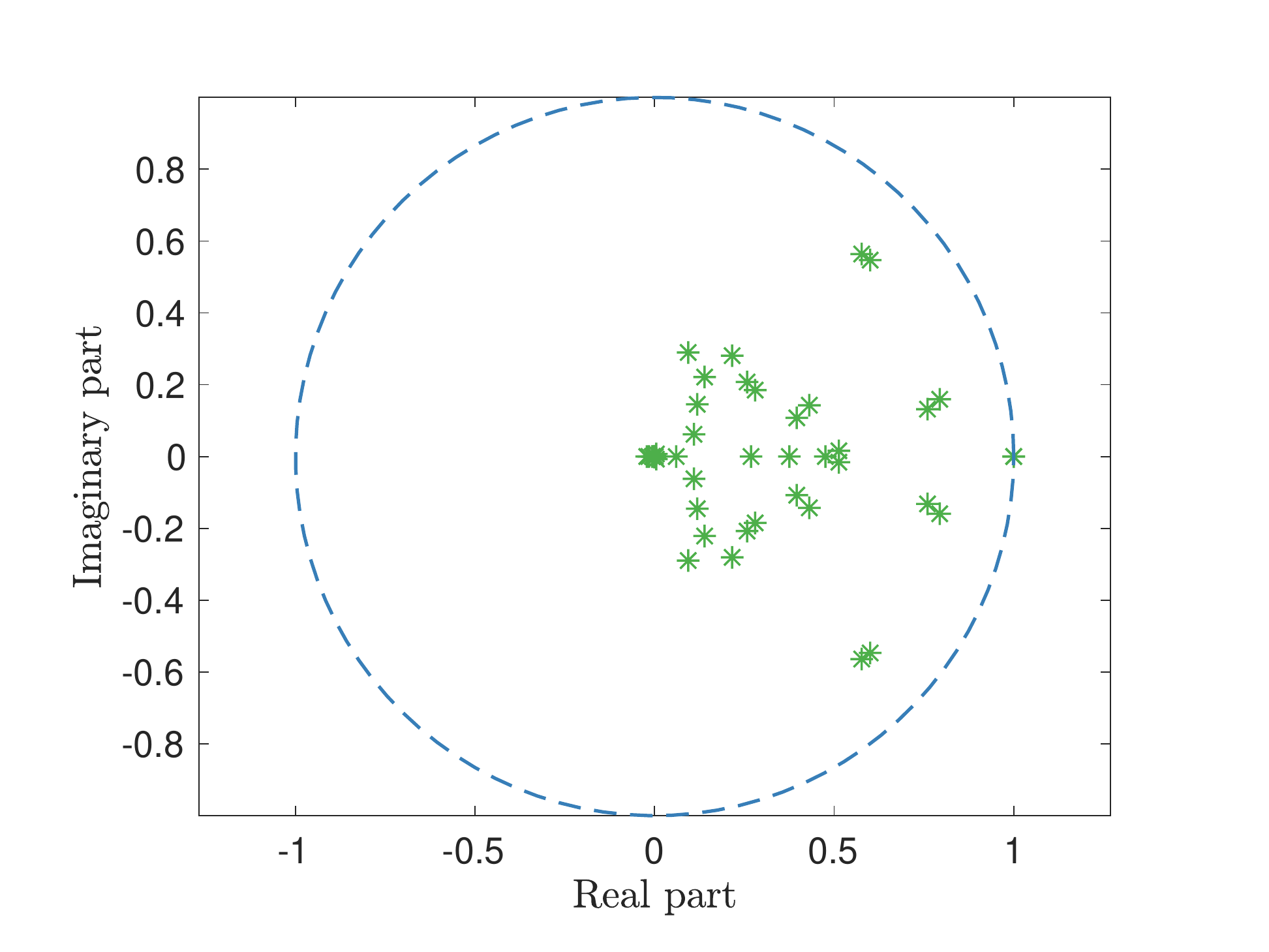}
\caption{Eigenvalues of the stitched Koopman operator ${\cal K}_{\cal S}$ corresponding to the bistable toggle switch system given in Eq. \eqref{eq:bistable_toggle_switch}.}
\label{fig:bistable_eigenvalues_stitched}
\end{center}
\end{figure}
\begin{figure}[h!]
\begin{center}
\subfigure[]{\includegraphics[width = 0.48 \linewidth]{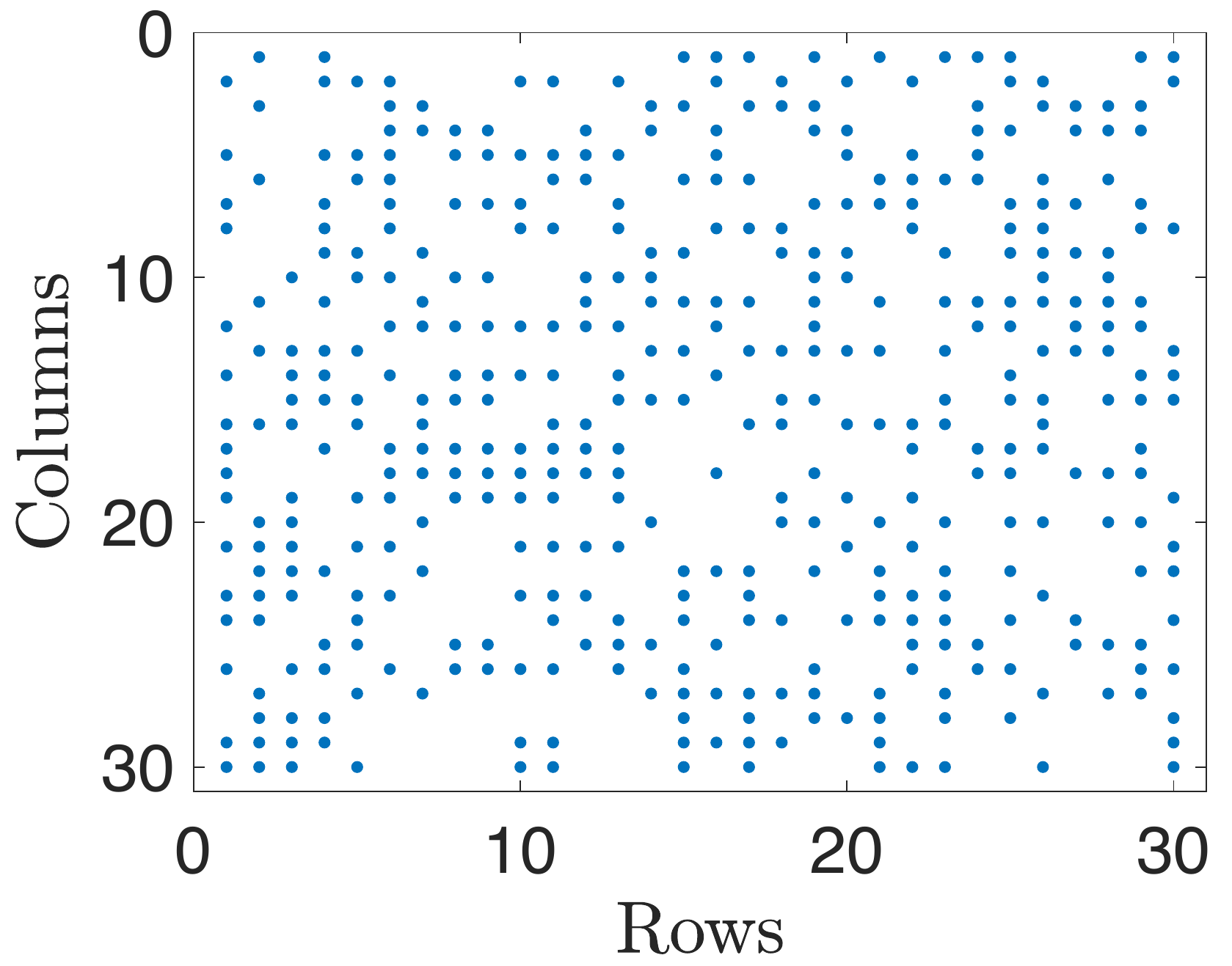}}
\subfigure[]{\includegraphics[width = 0.48 \linewidth]{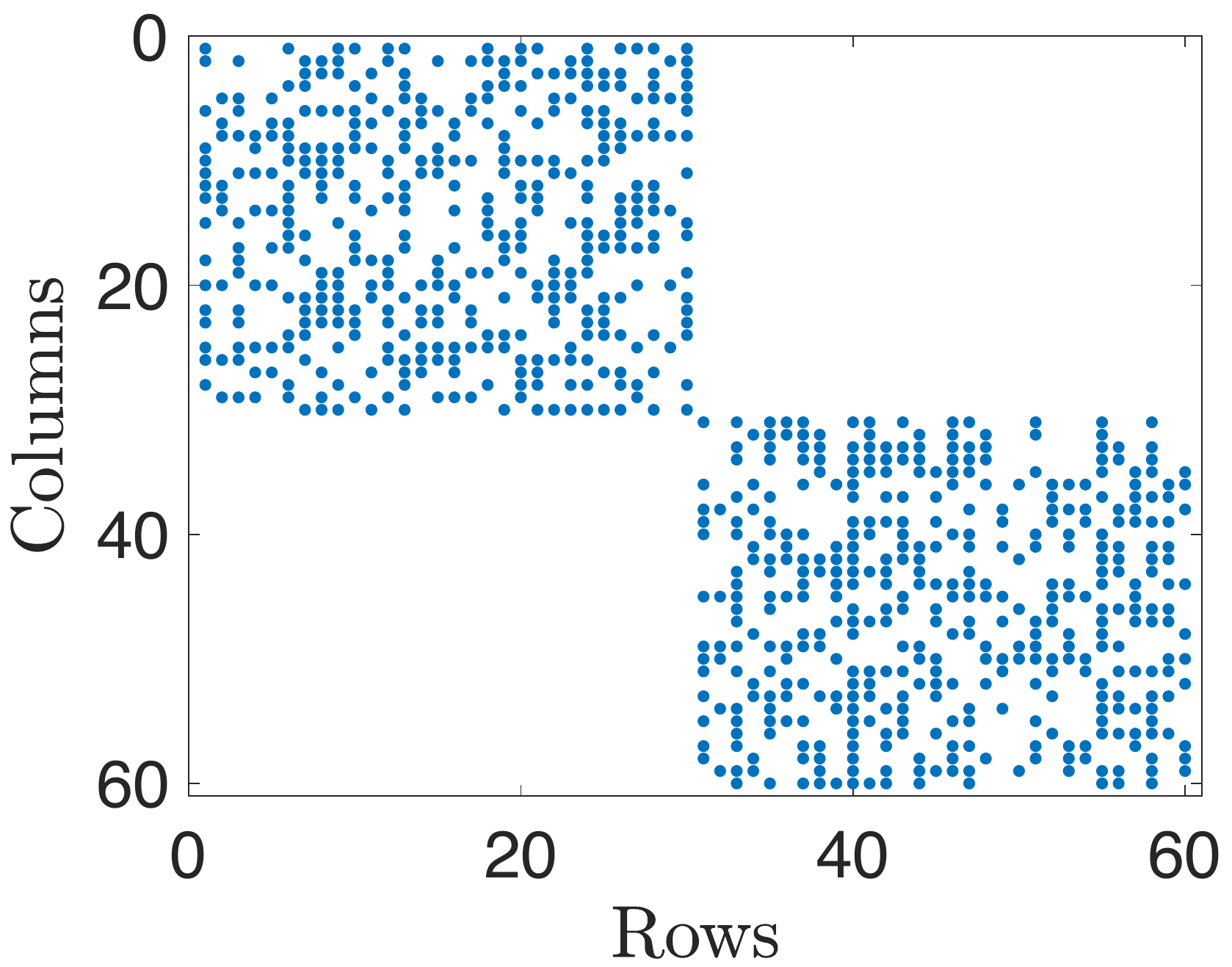}}
\caption{(a) Sparse structure of the Koopman operator ${\cal K}$ (b) Sparse structure of the Koopman operator ${\cal K}_{\cal S}$.}
\label{fig:bistable_Koopman_structures}
\end{center}
\end{figure}

It is important to validate ${\cal K}_{\cal S}$ to ensure that the dominant eigenvalue and eigenfunction pairs are rightly capturing the attractor sets on the state space. To show this, we computed the dominant eigenvalues and their respective eigenfunctions of ${\cal K}_{\cal S}$ which are shown in Fig. \ref{fig:bistable_inv_spaces_stitched}. It can be clearly seen from Fig. \ref{fig:bistable_inv_spaces_stitched} that each of the attractor sets on the state space are captured by the eigenfunctions associated with the first two leading eigenvalues of ${\cal K}_{\cal S}$. 
\begin{figure}[h!]
\begin{center}
\subfigure[]{\includegraphics[width = 0.5 \linewidth]{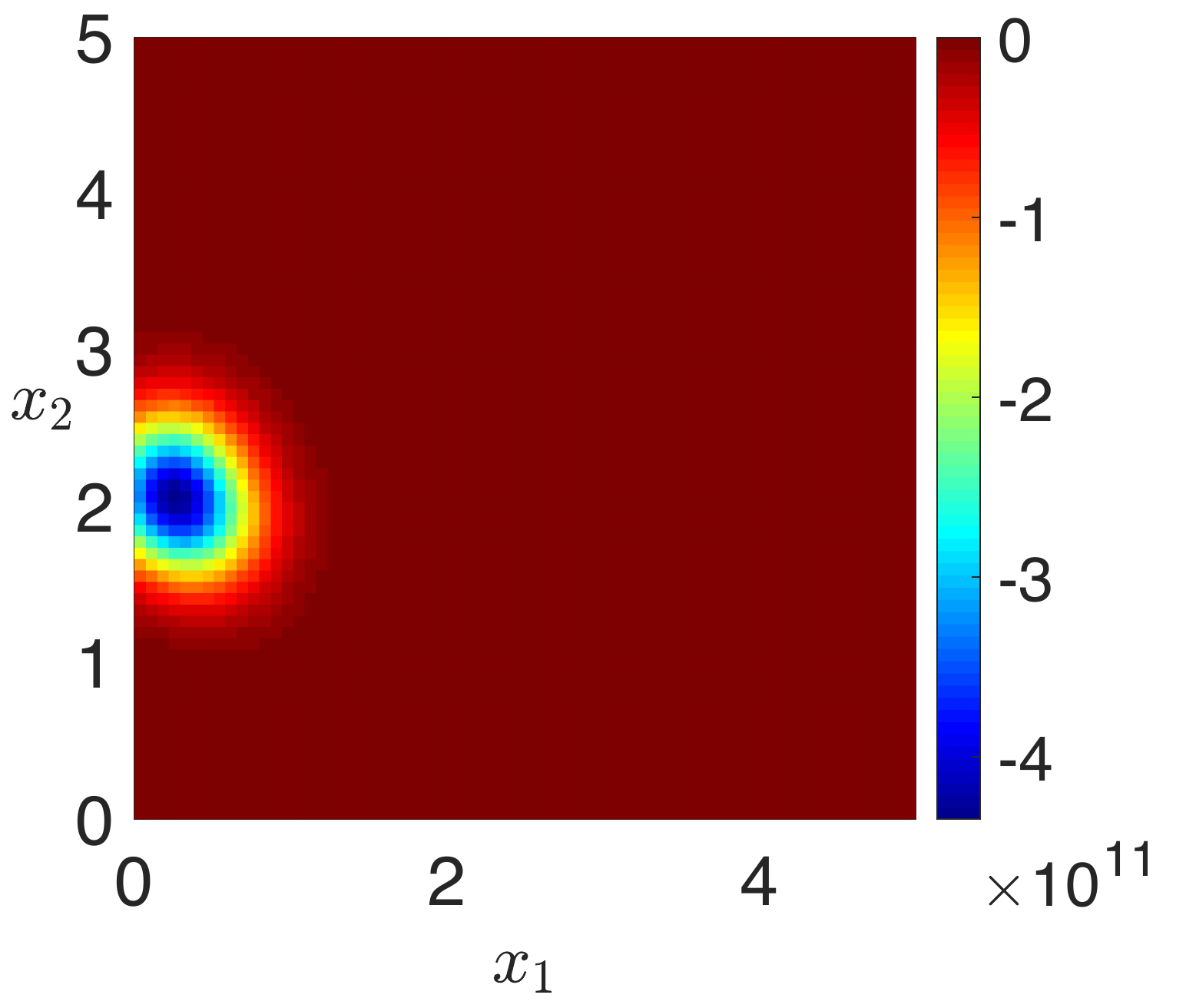}}
\subfigure[]{\includegraphics[width = 0.48 \linewidth]{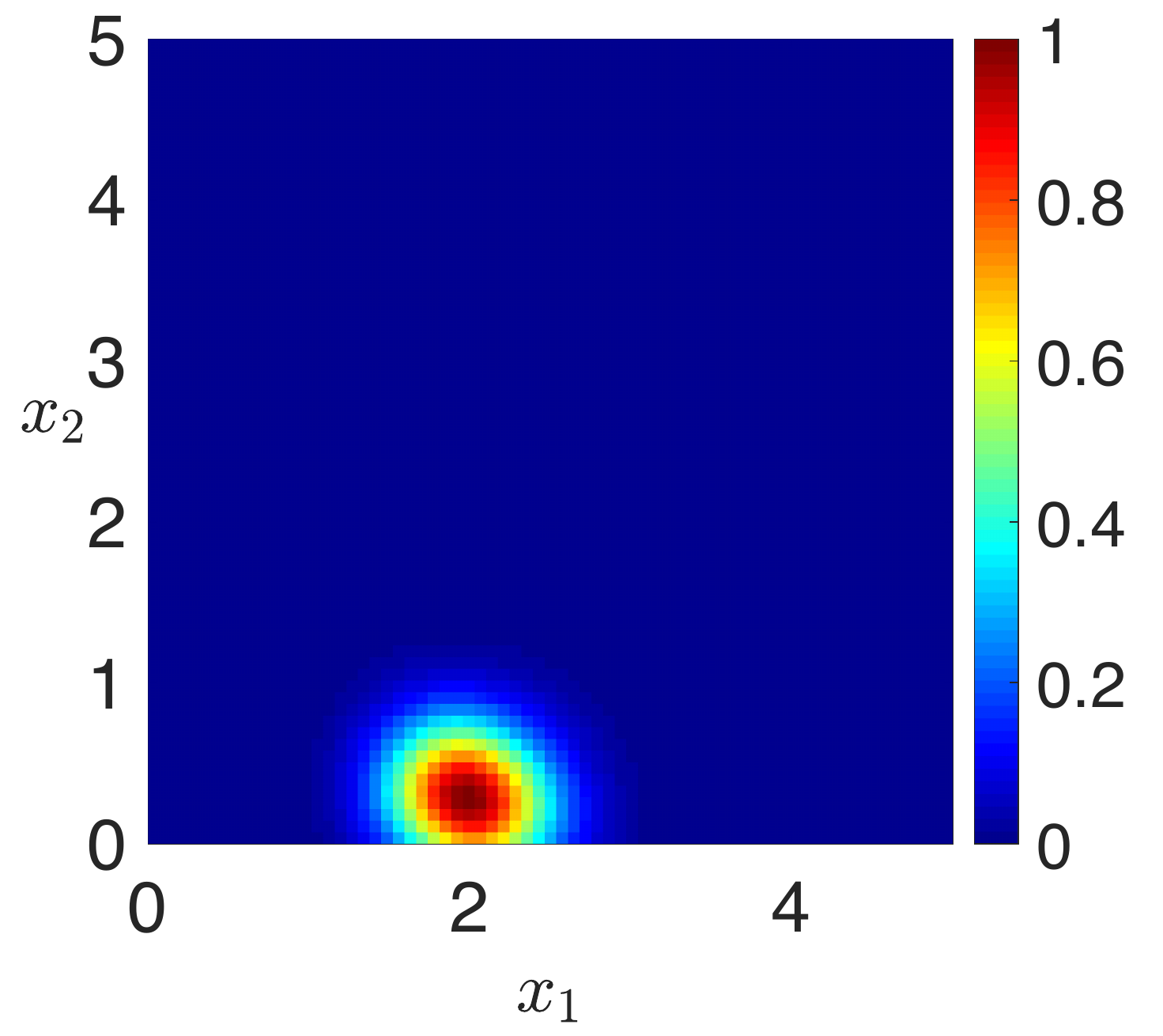}}
\caption{(a) and (b) Eigenvectors of ${\cal K}_{\cal S}$ associated with dominant eigenvalues $\lambda = 1$ on the state space. Eigenvectors of the stitched Koopman operator captures both the invariant sets of the state space.}
\label{fig:bistable_inv_spaces_stitched}
\end{center}
\end{figure}

Observe that, we have now computed two global Koopman operators, one assuming we have knowledge of the entire state space to obtain ${\cal K}$ and the other by stitching the local invariant subspaces to obtain ${\cal K}_{\cal S}$. Fig. \ref{fig:bistable_inv_spaces_stitched} shows the attractor sets on the state space identified using ${\cal K}_{\cal S}$ and moreover they also approximate well the invariant sets shown in Fig. \ref{fig:bistable_inv_spaces}. This validates the proposed approach of phase space stitching to compute the global Koopman operator from local Koopman operators. 
%
Next, we demonstrate the phase space stitching result on the following example. 
\subsection{Second-order System Example}
We consider a second-order dynamical system which is governed by the  dynamics:
\begin{equation}
\begin{aligned}
\dot{x}_1 = & x_1 - x_1 x_2 \\
\dot{x}_2 = & x_1^2 -2x_2
\end{aligned}
\label{eq:heart_dyn_system}
\end{equation}
The system \eqref{eq:heart_dyn_system} has 3 equilibrium points at $(\sqrt{2},1)$, $(-\sqrt{2},1)$ and $(0,0)$. It is seen that the origin is a saddle point and the other two equilibrium points are stable. The phase portrait of this system is shown in Fig. \ref{fig:phase_portrait_heart} and this system also has bistability similar to the bistable toggle switch system. 
\begin{figure}[h!]
\begin{center}
\includegraphics[width = 0.9 \linewidth]{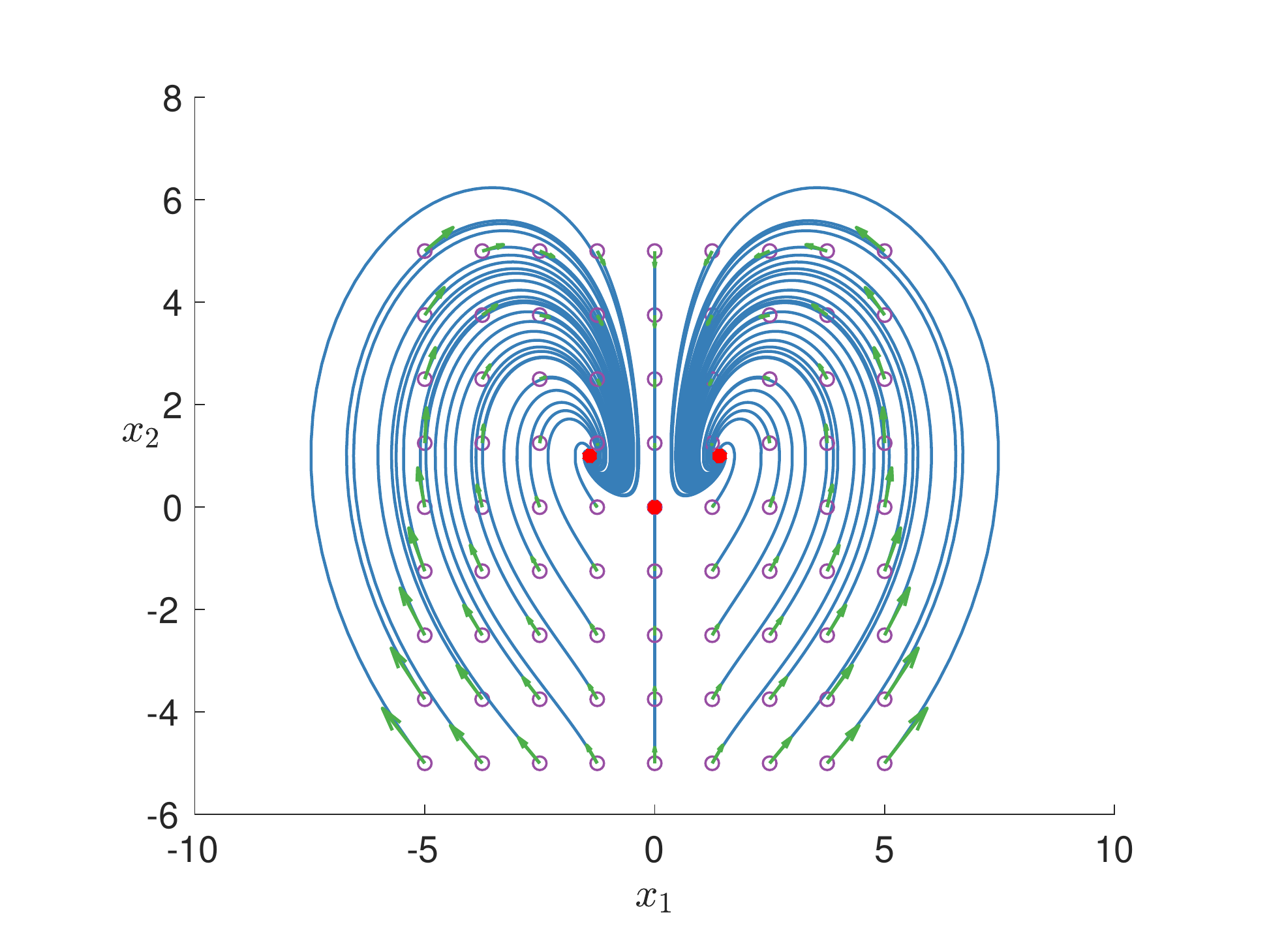}
\caption{Phase portrait corresponding to the system \eqref{eq:heart_dyn_system}. Circles indicate the initial conditions, green arrow indicate the direction of the vector field at the initial condition and the red dots are the equilibrium points.}
\label{fig:phase_portrait_heart}
\end{center}
\end{figure}

The Koopman operator trained on the complete state space data (which has $1000$ time points corresponding to $81$ initial conditions) is denoted by ${\cal K}$ where $30$ Gaussian radial basis functions with $\sigma = 0.4$ are used. The observable functions used for the bistable toggle switch are used for this system as well. Eigenvalues of the Koopman operator ${\cal K}$ can be seen in Fig. \ref{fig:eigenvalues_heart} and the eigenvectors corresponding to the dominant eigenvalues of ${\cal K}$ are shown in Fig. \ref{fig:inv_spaces_heart}. 
%
%
\begin{figure}[h!]
\begin{center}
\includegraphics[width = 0.8 \linewidth]{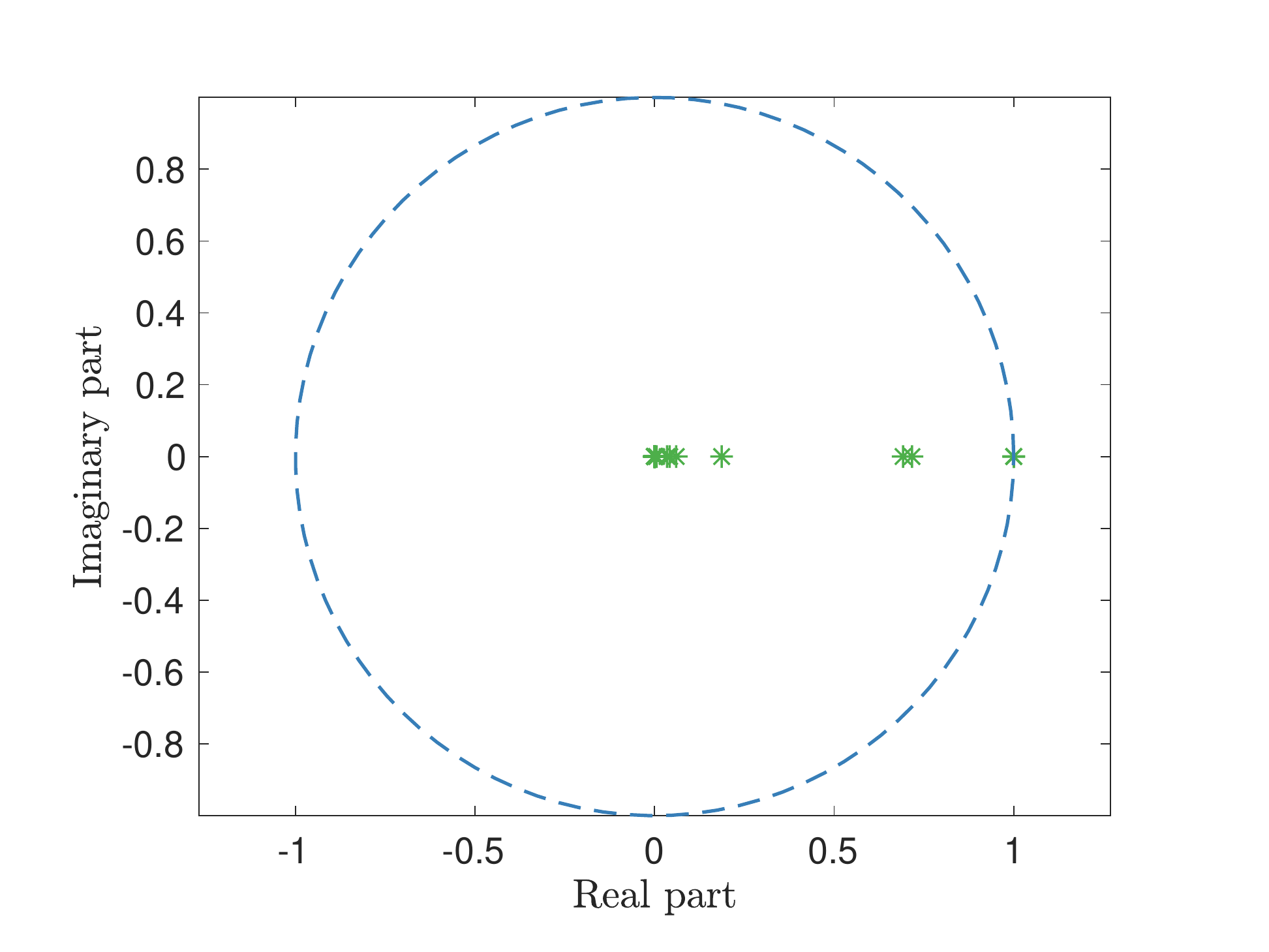}
\caption{Eigenvalues of the stitched Koopman operator, ${\cal K}$ corresponding to the system given in Eq. \eqref{eq:heart_dyn_system}.}
\label{fig:eigenvalues_heart}
\end{center}
\end{figure}

\begin{figure}[h!]
    \centering
    \subfigure[]{\includegraphics[width = 0.48 \linewidth]{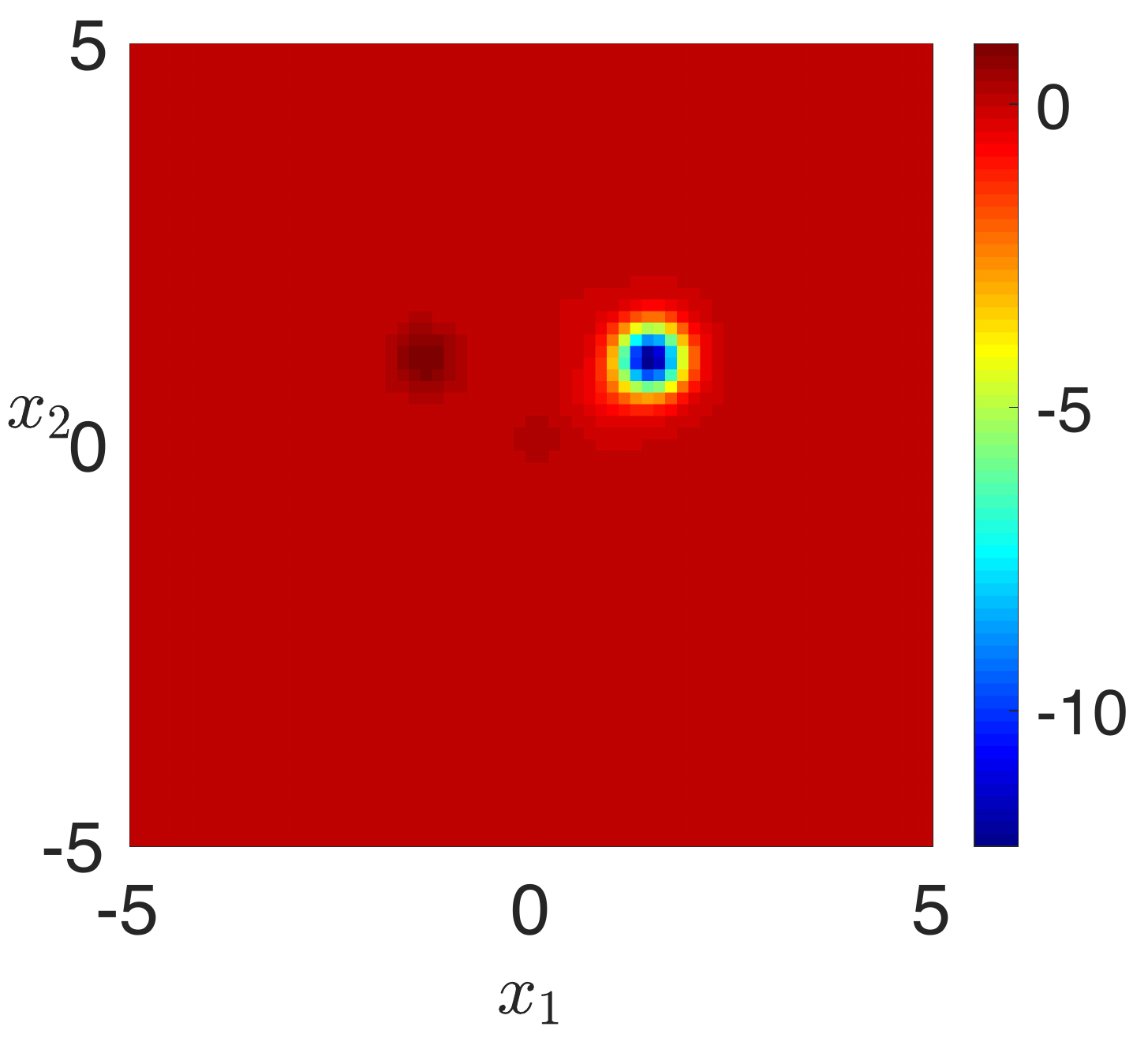}}
    \subfigure[]{\includegraphics[width = 0.5 \linewidth]{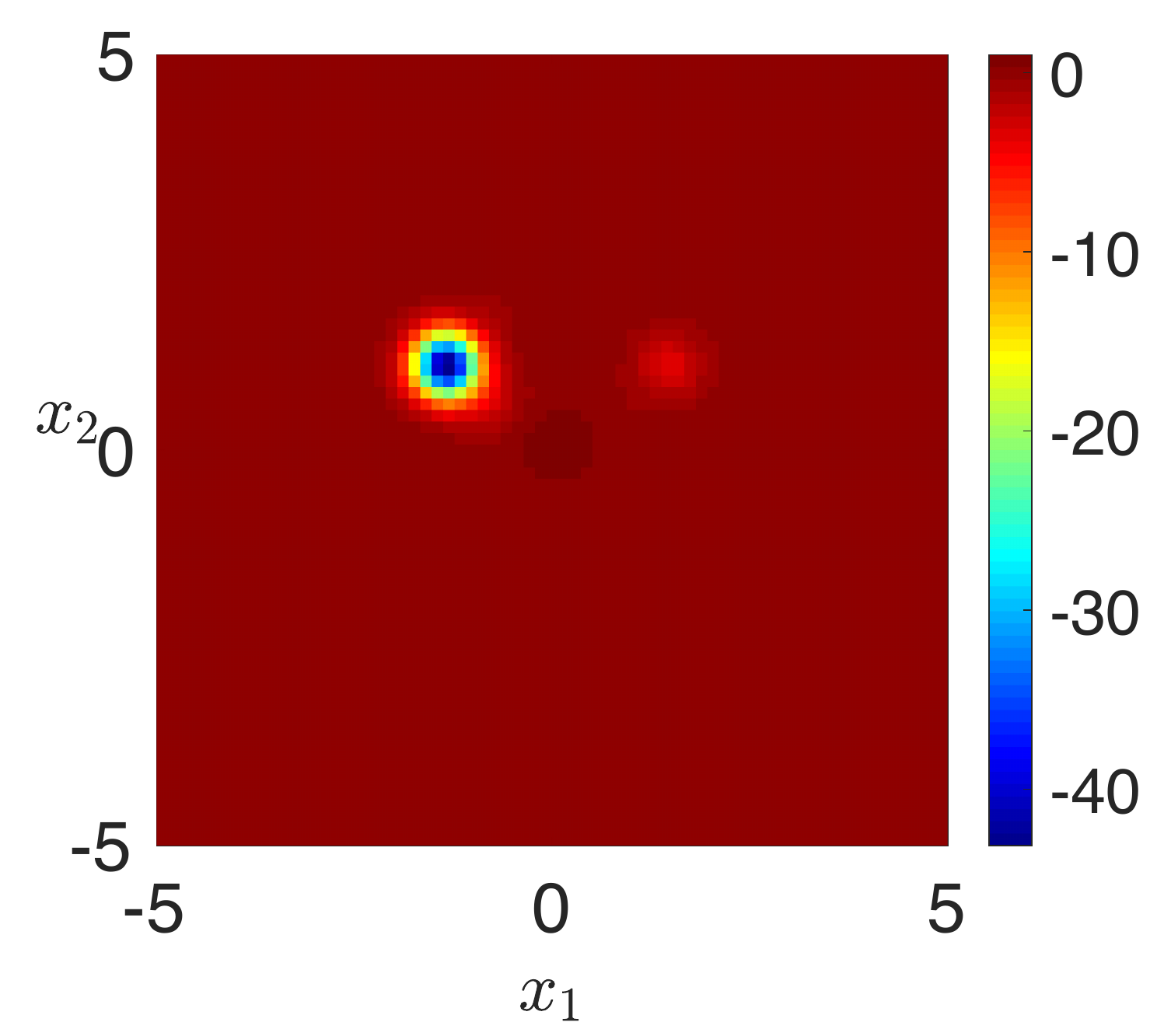}}
    \subfigure[]{\includegraphics[width = 0.5 \linewidth]{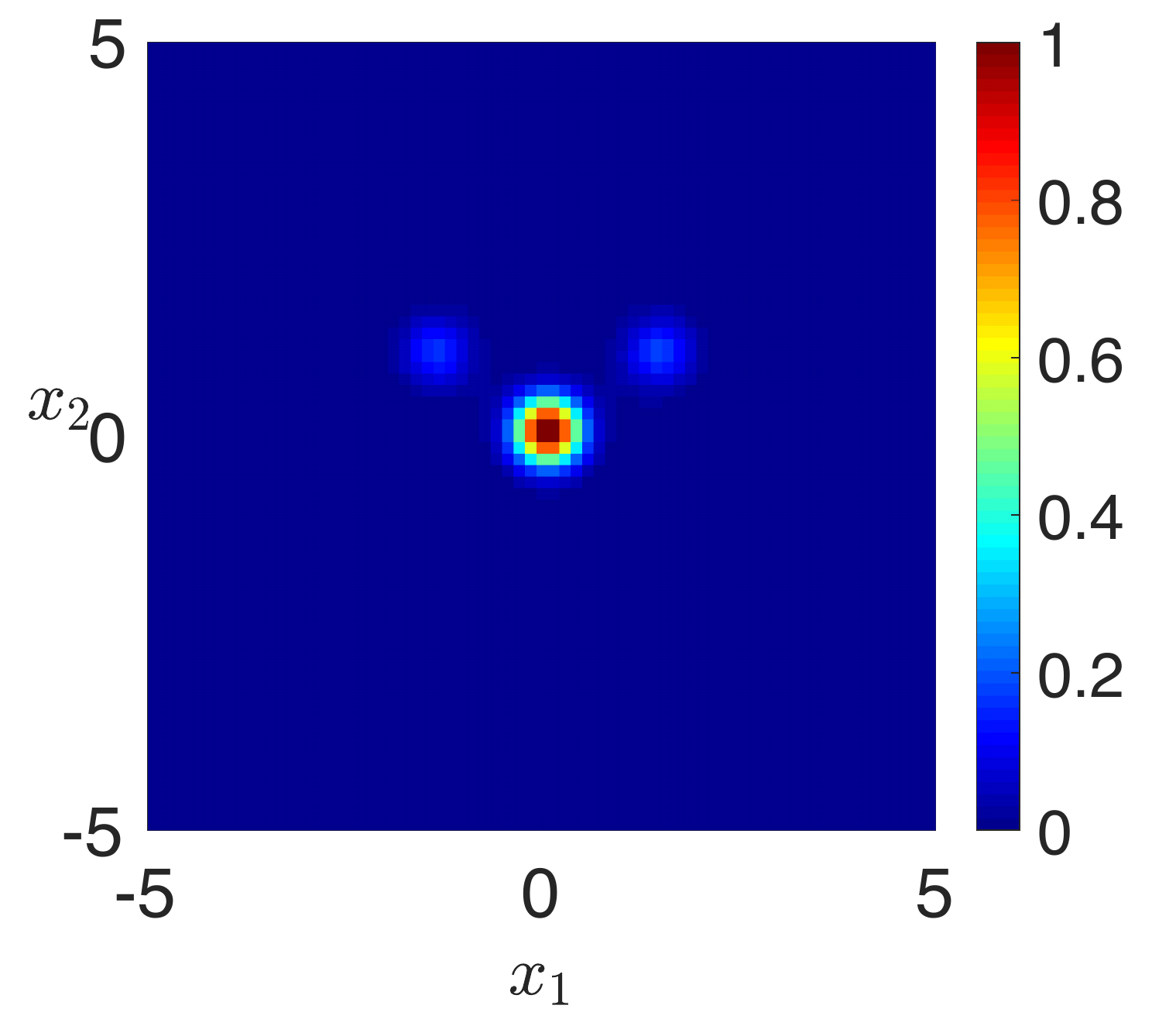}}
    \caption{Eigenvectors corresponding to the dominant eigenvalues of ${\cal K}$ on state space. (a) and (b) stable equilibrium points and (c) saddle point. }
    \label{fig:inv_spaces_heart}
\end{figure}

Next, assume that no knowledge of the complete phase space is known and the knowledge about the local invariant sets are only known. Then, the corresponding local Koopman operators are denoted by ${\cal K}_l$ and ${\cal K}_r$. Furthermore, the stitched Koopman operator with respect to these local Koopman operators is given by ${\cal K}_{\cal S}$. 
The eigenvalues and the sparse structure of the stitched Koopman operator when compared to ${\cal K}$ are shown in Fig. \ref{fig:eigenvalues_heart_stitched} and Fig. \ref{fig:heart_Koopman_structures} respectively. 
\begin{figure}[h!]
\begin{center}
\includegraphics[width = 0.8 \linewidth]{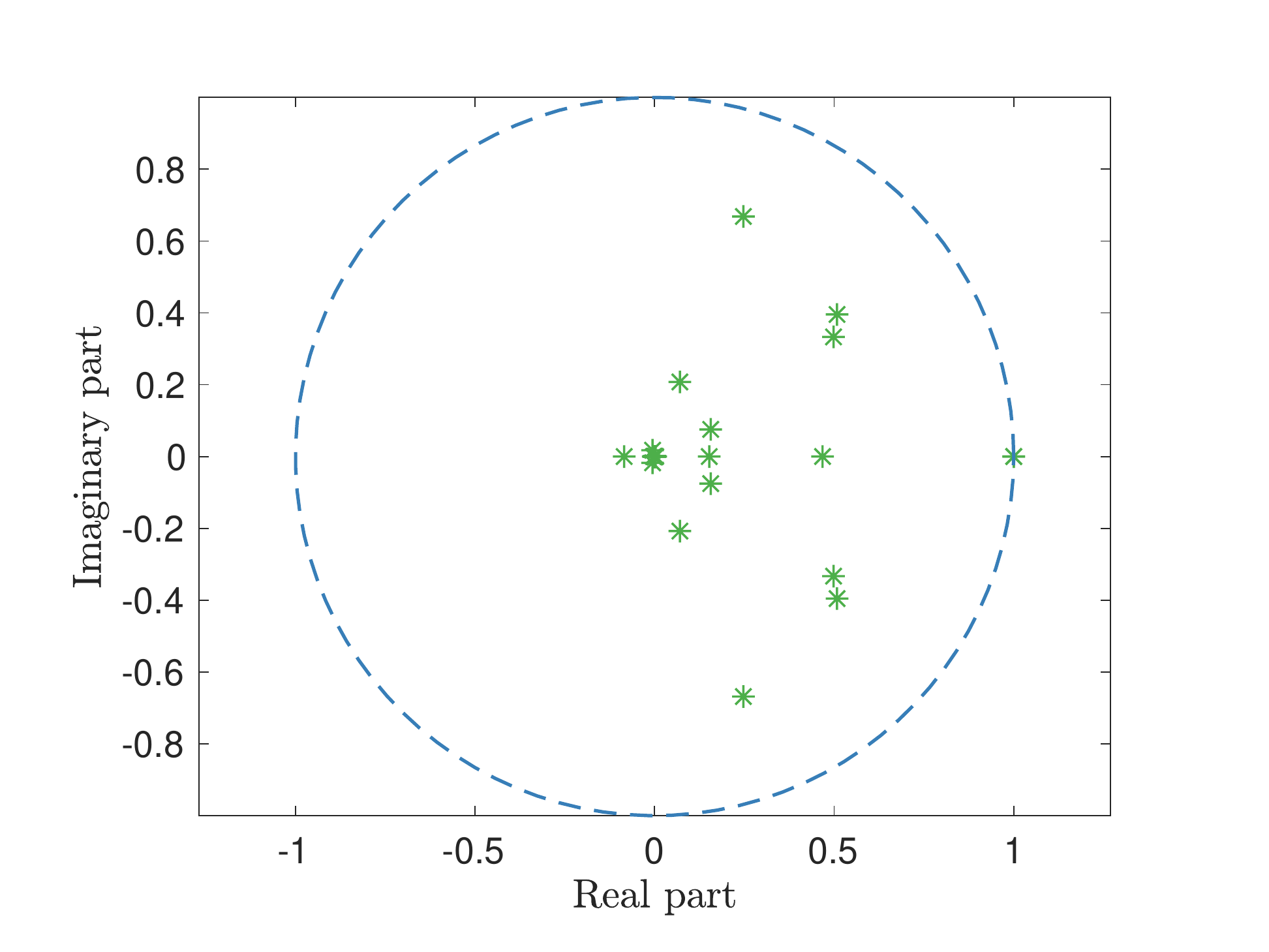}
\caption{Eigenvalues of the stitched Koopman operator ${\cal K}_{\cal S}$ corresponding to the system \eqref{eq:heart_dyn_system}.}
\label{fig:eigenvalues_heart_stitched}
\end{center}
\end{figure}
\begin{figure}
\begin{center}
\subfigure[]{\includegraphics[width = 0.48 \linewidth]{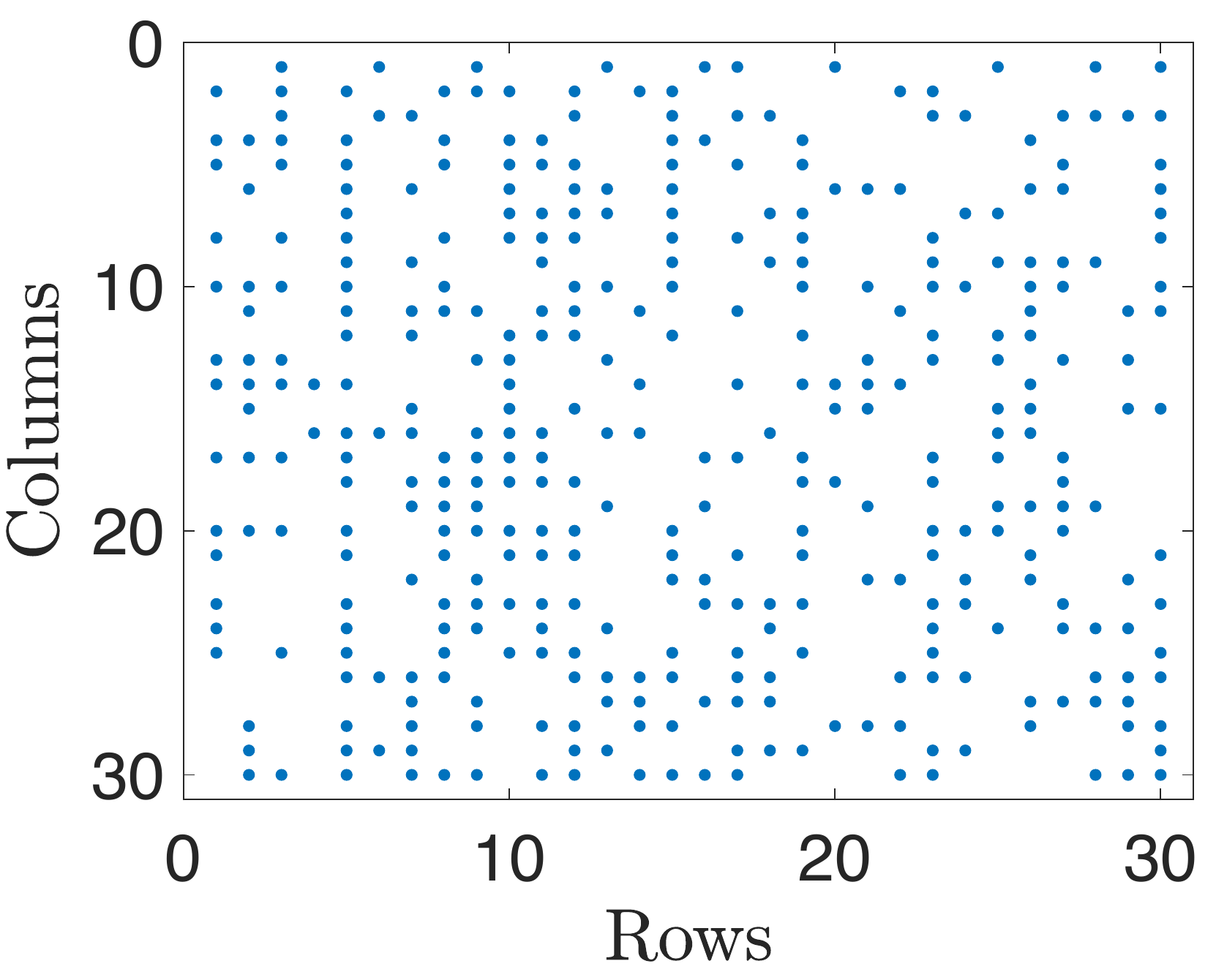}}
\subfigure[]{\includegraphics[width = 0.48 \linewidth]{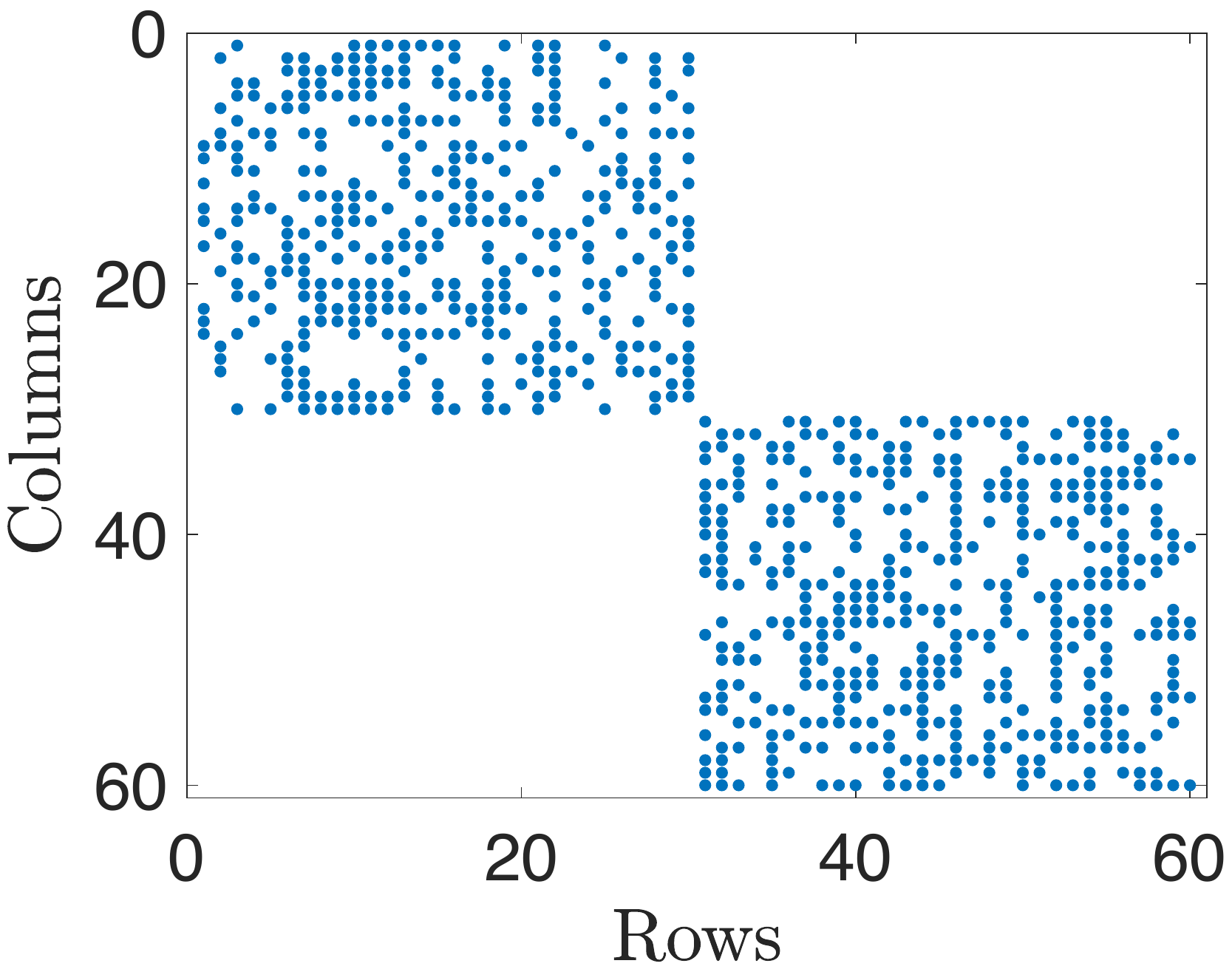}}
\caption{(a) Sparse structure of the Koopman operator ${\cal K}$ (b) Sparse structure of the Koopman operator ${\cal K}_{\cal S}$.}
\label{fig:heart_Koopman_structures}
\end{center}
\end{figure}

The eigenvector plots corresponding to the dominant eigenvalues of ${\cal K}_{\cal S}$ are shown in Fig. \ref{fig:heart_inv_spaces_stitched} (a) and (b). It can be seen that the stitched Koopman operator identifies the attractor sets in the state space. 
\begin{figure}[h!]
\begin{center}
\subfigure[]{\includegraphics[width = 0.50 \linewidth]{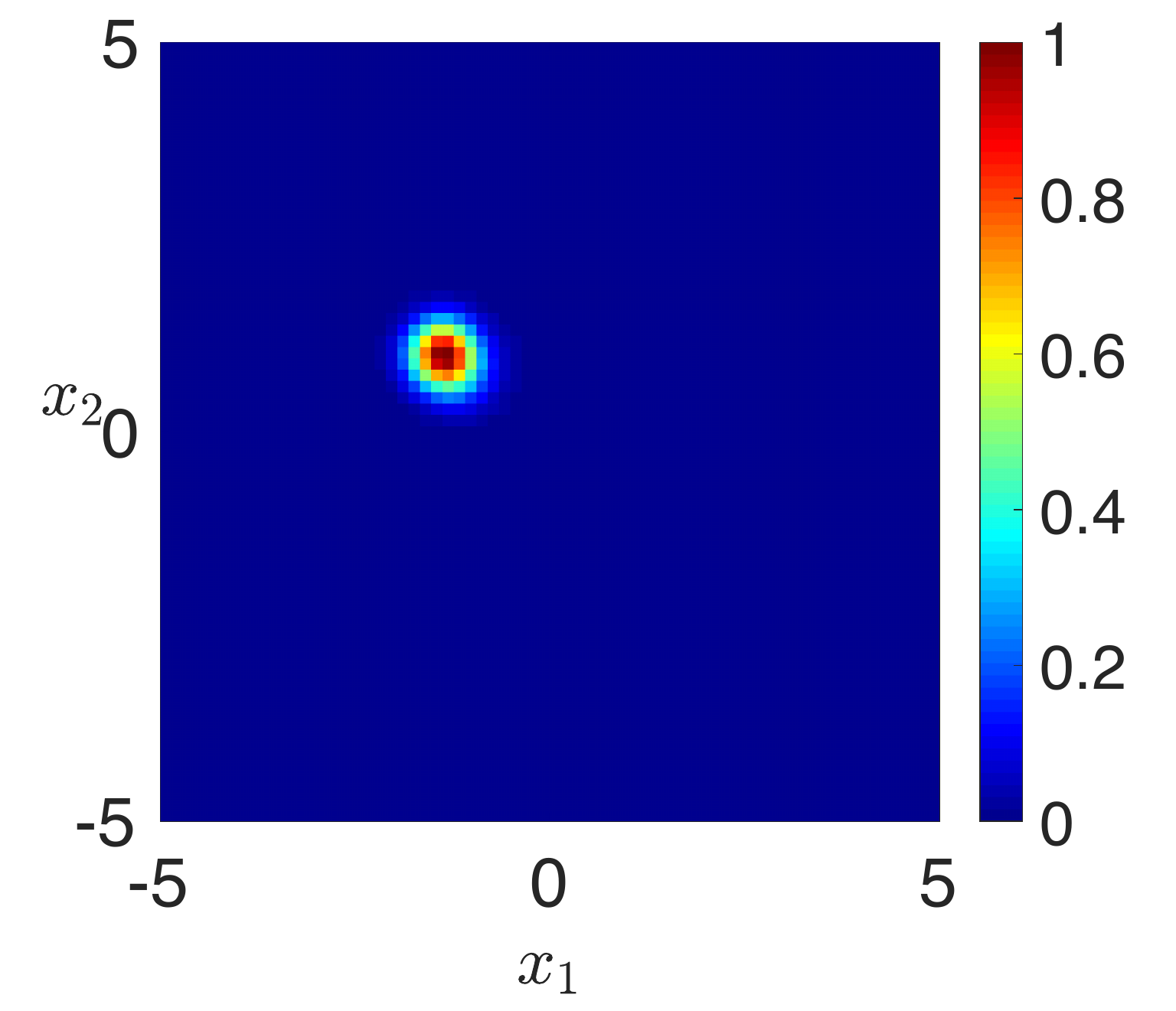}}
\subfigure[]{\includegraphics[width = 0.47 \linewidth]{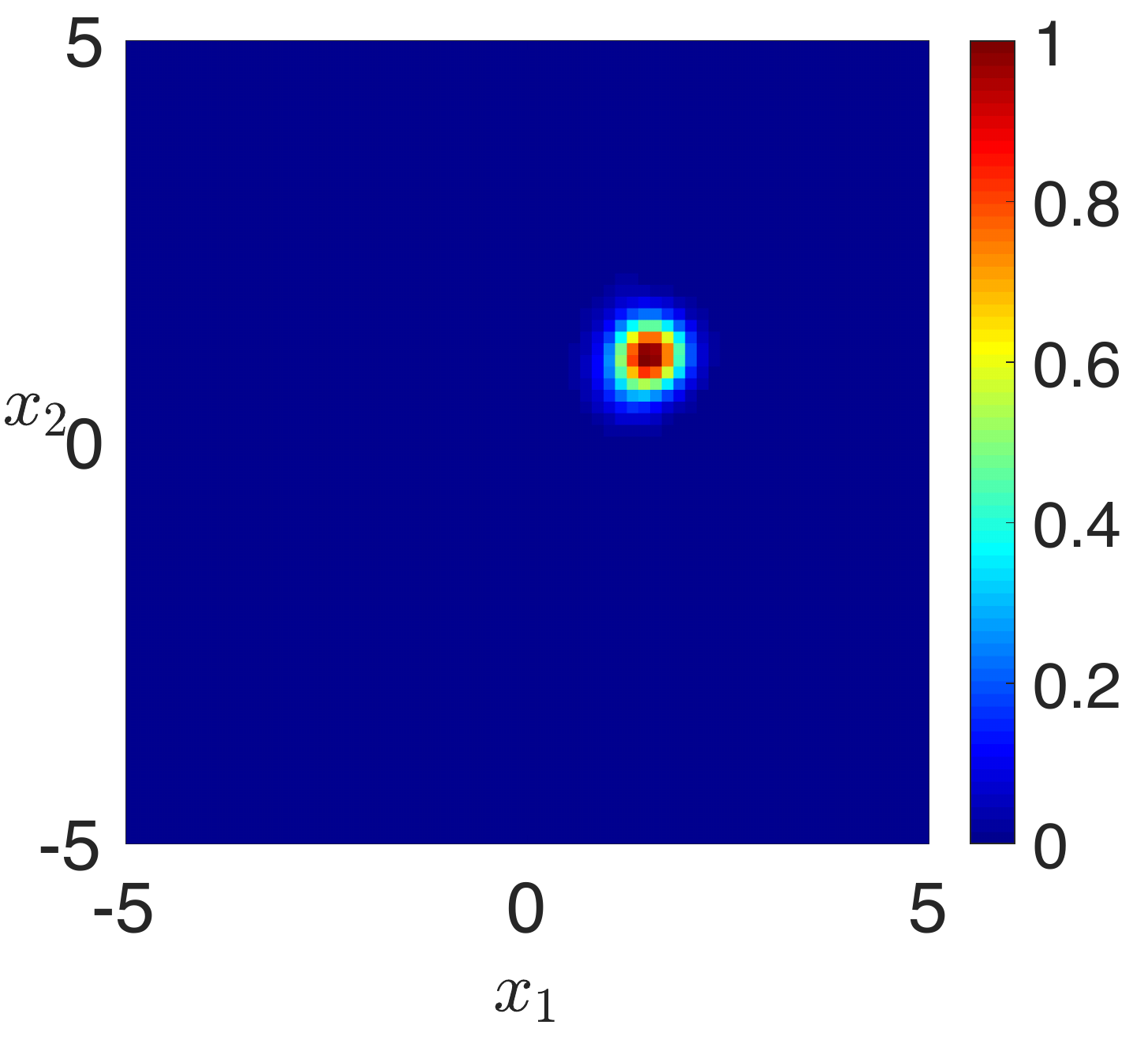}}
\caption{(a) and (b) Eigenvectors of ${\cal K}_{\cal S}$ associated with dominant eigenvalues $\lambda = 1$ on the state space. Eigenvectors of the stitched Koopman operator captures both the invariant sets of the state space.}
\label{fig:heart_inv_spaces_stitched}
\end{center}
\end{figure}
Therefore, the stitched Koopman operator describes the global behavior of the system and identifies the attractor sets of the phase space in comparison to ${\cal K}$ (which me may not have access to in real experiments. 

\section{Conclusion}
\label{sec:conclusion}
This works deals with the global phase space exploration of a dynamical system using data-driven operator theoretic approaches. Using the spectral properties of the Koopman operator, conditions are presented to identify multiple invariant subspaces of a dynamical system given the time-series data of a dynamical system. Furthermore, a phase space stitching result is developed which results in the global evolution of the system by considering the behaviors of the system locally. Finally the developed results are illustrated on two systems: a bistable toggle switch and a second-order example. For the validation purpose of the proposed framework, the invariant subspaces identified by the phase space stitched Koopman operator are compared against the global Koopman operator. This work has the potential to experimentally discover quiescent phenotypes of a biological system which forms the future efforts of the scope of this work.

\section{Acknowledgments}
The authors would also like to thank Igor Mezic, Nathan Kutz, Erik Bollt, Robert Egbert, and Umesh Vaidya for stimulating conversations.  Any opinions, findings and conclusions or recommendations expressed in this material are those of the author(s) and do not necessarily reflect the views of the Defense Advanced Research Projects Agency (DARPA), the Department of Defense, or the United States Government. This work was supported partially by a Defense Advanced Research Projects Agency (DARPA) Grant No. DEAC0576RL01830 and an Institute of Collaborative Biotechnologies Grant.

\bibliographystyle{IEEEtran}
\bibliography{Refs}
\end{document}